\documentclass[12pt,reqno]{amsart}

\usepackage{epsfig}
\usepackage{amsmath, amsthm, amsfonts, amssymb}
\usepackage{graphicx}

\numberwithin{equation}{section}

\newcommand{\D}{{\mathbb D}}

\newcommand{\R}{{\mathbb R}}
\newcommand{\C}{{\mathbb C}}

\newcommand{\Z}{{\mathbb Z}}

\newcommand{\K}{{\mathbb K}}
\renewcommand{\d}{\partial}

\newcommand{\f}{\varphi}

\newtheorem{theo}{{\sc \bf Theorem}}[section]
\newtheorem{cor}[theo]{{\sc \bf Corollary}}
\newtheorem{lem}[theo]{{\sc \bf Lemma}}
\newtheorem{prop}[theo]{{\sc \bf Proposition}}

\newenvironment{defin}{\medskip\noindent{\it Definition:\/} }{\medskip}

\begin{document}

\title[Dirac type operators on the quantum solid torus]{Dirac type operators on the quantum solid torus with global boundary conditions}

\author{Slawomir Klimek}
\address{Department of Mathematical Sciences,
Indiana University-Purdue University Indianapolis,
402 N. Blackford St., Indianapolis, IN 46202, U.S.A.}
\email{sklimek@math.iupui.edu}

\author{Matt McBride}
\address{Department of Mathematics and Statistics,
Mississippi State University,
175 President's Cir., Mississippi State, MS 39762, U.S.A.}
\email{mmcbride@math.msstate.edu}

\thanks{}

\date{\today}

\begin{abstract}
We define a noncommutative space we call the quantum solid torus. It is an example of a noncommutative manifold with a noncommutative boundary.
We study quantum Dirac type operators subject to Atiyah-Patodi-Singer like boundary conditions on the quantum solid torus. We show that such operators have compact inverse, which means that the corresponding boundary value problem is elliptic.
\end{abstract}

\maketitle
\section{Introduction}
In this paper we are continuing our study of quantum analogs of Dirac type operators on manifolds with boundary and global boundary conditions of Atiyah, Patodi, and Singer type \cite{APS}. Despite similarities with our previous papers on the subject, the problem studied here is much more complex. We consider a higher dimensional space, namely three dimensional solid torus, and its quantization. The resulting space is noncommutative and its boundary is also noncommutative. Dirac type operators are matrix valued operators and the Atiyah-Patodi-Singer condition cannot be generalized straightforwardly.

One of the main objectives of our previous work was to find appropriate non-commutative analogs of Dirac type operators on planar domains with boundary and to investigate their functional analytic properties.   Some examples of Dirac type operators in simple domains, such as the disk, annulus, and punctured disk were described in \cite{KM1} and \cite{KM2}.   These papers exhibited strong similarities in the setup and the results between the commutative and quantum cases.   Also in those papers the global boundary condition imposed on Dirac-like operators was essentially the classical APS boundary condition.  In reference \cite{KM4} we discussed Dirac type operators on the solid torus, in the commutative sense, with a different nonlocal boundary condition, that was inspired by \cite{MS}. This is because the natural metric on the disk does not have the  product structure near boundary, required for APS theory. This was not a problem for two dimensional domains but becomes an issue in dimension three. In this paper we follow up the analysis \cite{KM4} of the operators on the commutative solid torus, with same type of investigation for their quantum analogs.

Instead of constructing a quantum analog of the non-local boundary condition that was used in \cite{KM4}, we consider a large class of boundary conditions that yield desired analytical properties of the parametrices of the quantum Dirac operators. The boundary condition of \cite{KM4} required extending functions in the domain beyond the boundary of the solid torus, which makes sense geometrically in the classical case.   However in the quantum analog, there are obvious obstructions, for example what does we mean by the ``outside'' of the boundary of the quantum solid torus.   This will be addressed in our future work.

The solid torus is the product of a disk and a circle. In quantum case  we take a noncommutative disk and then a twisted product of it with the unit circle to produce the quantum solid torus.	 The operators we consider respect this decomposition. Unlike our previous work on d-bar operators on the quantum disk \cite{KM1}, we consider here a bigger variety of related Dirac like operators on the disk. They yield a wide class of operators on the solid torus.   The crux of the analysis presented here  is the proof that the parametrices of our Dirac type operators are compact operators, like it was shown in \cite{KM4} that the parametrix of the classical Dirac type operator was also compact.  

The paper is organized as follows.   In section 2 the quantum solid torus is introduced and discussed.  This is followed by section 3 on partial Fourier series on the quantum solid torus. The relevant Hilbert spaces are also defined there. Section 4  introduces our Dirac type operators, while section 5 contains a discussion of some properties of their coefficients.   The boundary conditions are defined in section 6, based on properties of special solutions introduced there.   The computation of the parametrix of the Dirac type operator is described in section 7.   Finally, section 8 contains analysis of a parametrix culminating in the proof of the main theorem on compactness. 

\section{Non-commutative solid Torus}
In this section we define our version of the quantum solid torus.   From the topological point of view such a quantum torus is a noncommutative $C^*-$algebra which has a structure similar in some ways to the ordinary solid torus.  The idea here is simple:  the solid torus can be thought of as the product of the disk and the circle.   To obtain a quantum solid torus we take the quantum disk of \cite{KL} and take its twisted product with the circle.  The result can be described either as a suitable crossed product algebra or as a universal $C^*-$algebra with generators and relations.  We proceed to describe them now.

\begin{defin}
Let $A_\theta$ be the universal $C^*-$algebra with generators $U$ and $V$ such that $U$ is an isometry i.e. $U^*U=I$, $V$ is a unitary i.e. $V^*V=VV^*=I$ and such that they satisfy the commutation relation 
\begin{equation*}
VU = e^{2\pi i\theta} UV.
\end{equation*}
\end{defin}

Let $\{e_k\}$ be the canonical basis for $\ell^2(\Z_{\ge0})$ and let $W$ be the unilateral shift, i.e. $We_k = e_{k+1}$.   Let $\mathcal{T}$ be the $C^*-$algebra generated by $W$. The algebra $\mathcal{T}$ is called the Toeplitz algebra and by Coburn's theorem it is the universal $C^*-$algebra with generator $W$ satisfying the relation $W^*W=I$, i.e. an isometry.  Reference \cite{KL} shows that this algebra can be thought of as a quantum unit disk.   It's structure is described by a short exact sequence, namely:

\begin{equation}\label{TExact}
0\longrightarrow \mathcal{K} \longrightarrow \mathcal{T} \longrightarrow C(S^1) \longrightarrow 0.
\end{equation}
where $ \mathcal{K} $ is the ideal of compact operators in $\ell^2(\Z_{\ge0})$.  In fact $\mathcal{K}$ is the commutator ideal of the algebra $\mathcal{T}$.

\begin{prop}
For any $\theta$ we have the following isomorphism of algebras: $A_\theta \cong \mathcal{T} \rtimes_\theta \Z$ where the crossed product on the right hand side is defined, for $n\in\Z$, by the automorphisms $\Theta_n(U) = e^{2\pi in\theta}U$ for some $\theta\in[0,2\pi)$.
\end{prop}
Note that because of the universality of the Toeplitz algebra, we only need to define an automorphism on its generator and check that it satisfies the defining relation.

\begin{proof}
By definition, a representation $\pi$ of the universal $C^*-$algebra $A_\theta$, consists of a Hilbert space, $H$ and bounded operators  $\pi(U)$ and $\pi(V)$ on $H$, that satisfy the defining relations of $A_\theta$.  Then $A_\theta$ is the completion of the algebra of polynomials $a$ in $U$ and $V$ and their adjoints with respect to the maximal norm, i.e.
\begin{equation*}
\|a\|_{\textrm{max}} = \underset{\pi}{\textrm{sup }}\|\pi(a)\|.
\end{equation*}
On the other hand, the covariant representation of the dynamical system $(\Z, \Theta_n, \mathcal{T})$ consists of a Hilbert space $X$, a bounded representation $\rho$ of $\mathcal{T}$ on $X$, and a unitary operator $\rho(V)$ on $X$, such that  the following commutation relation holds:
$$\rho(\Theta_1(b)) = \rho(V)\rho(b)\rho(V)^*,$$ 
for any $b\in\mathcal{T}$. 
By the universality of the Toeplitz algebra, its representation $\rho$ is completely determined by an isometry $\rho(U)$ i.e. $\rho(U)\rho(U)^* = I$.   Moreover, the above commutation relation with $\rho(V)$ becomes the following: 
$$e^{2\pi i\theta}\rho(U) = \rho(V)\rho(U)\rho(V)^*.$$   
Consequently the crossed product $\mathcal{T}\rtimes_\theta \Z$ is obtained by completion of the algebra of polynomials $a$ in $U$ and $V$ and their adjoints with respect to the maximal norm, i.e.
\begin{equation*}
\|a\|_{\textrm{max}} = \underset{\rho}{\textrm{sup }}\|\rho(a)\|.
\end{equation*}
So we see that the two concepts are identical, establishing an isomorphism between the algebras, and completing the proof.
\end{proof}

By the amenability of $\Z$, the crossed product of the Toeplitz algebra with $\Z$ is equal to the reduced crossed product of the Toeplitz algebra with $\Z$.   Consequently, the norm of an element of the crossed product can computed from a single representation by combining a faithful representation of the Toeplitz algebra with the left regular representation of $\Z$.  Explicitly, choosing the defining representation of the Toeplitz algebra, we can view the reduced crossed product as the $C^*-$algebra generated by the operators described below acting in the Hilbert space $\ell^2(\Z,\ell^2(\Z_{\ge0})) = \ell^2(\Z)\otimes\ell^2(\Z_{\ge0}) = \ell^2(\Z_{\ge0}\times\Z)$, where the equality of Hilbert spaces follows from \cite{RS}. 

\begin{cor}\label{Athetaconcrete}
Let $\{e_{k,l}\}$ be the canonical basis in $\ell^2(\Z_{\geq 0}\times\Z)$.   Defining two operators: $Ue_{k,l} = e^{-2\pi il\theta}e_{k+1,l}$ and $Ve_{k,l} = e_{k,l+1}$, we have $A_\theta = C^*(U,V)$.
\end{cor}

For $\theta\in[0,2\pi)$ let $T_\theta^2$ be the two dimensional quantum torus.   In other words, $T_\theta^2$ is the universal $C^*-$algebra with two unitary generators $u$ and $v$ such that $vu = e^{2\pi i\theta}uv$, see for example \cite{Davidson}. The following proposition describes the structure of $A_\theta$.

\begin{prop}
We have the following short exact sequence:

\begin{equation*}
0\longrightarrow \mathcal{K}\otimes C(S^1)\longrightarrow A_\theta\longrightarrow T_\theta^2\longrightarrow 0.
\end{equation*}
\end{prop}

\begin{proof}
Consider the crossed product of the short exact sequence (\ref{TExact}) for the Toeplitz algebra with $\Z$. We get:

\begin{equation*}
0\longrightarrow \mathcal{K}\rtimes_\theta\Z \longrightarrow \mathcal{T}\rtimes_\theta\Z \longrightarrow C(S^1)\rtimes_\theta\Z \longrightarrow 0.
\end{equation*}
It is well known that $T_\theta^2\cong C(S^1)\rtimes_\theta\Z$, see \cite{Davidson}.   Moreover, it is easy to see that $\mathcal{K}\rtimes_\theta\Z\cong\mathcal{K}\otimes C(S^1)$, see \cite{Williams} for the details. Explicitly, the cross product $\mathcal{K}\rtimes_\theta\Z$ can be ``untwisted" by noticing that, in the tensor product of Hilbert spaces $\ell^2(\Z)\otimes\ell^2(\Z_{\ge0})$, if $a$ is compact in $\ell^2(\Z_{\ge0})$ then:  
\begin{equation*}
(I\otimes a)(V')^n=(I\otimes e^{-2\pi i\K\theta}a)\cdot(V\otimes I)^n.
\end{equation*}
The operators above are $\K e_{k,l} = ke_{k,l}$ and $V'e_{k,l} = e^{-2\pi ik\theta}e_{k,l+1}$.  In addition define the operator $U'e_{k,l} = e_{k+1,l}$, then we have $A_\theta = C^*(U',V')$ by universality, while $e^{-2\pi i\K\theta}a$ is still compact.
Combination of all those facts above yields the above short exact sequence.
\end{proof}

In view of the results about $A_\theta$ contained in this section it natural to call this algebra the quantum solid torus.

\section{Fourier series}

The purpose of this section is to introduce a partial Fourier transform on the quantum torus.  As stated in Corollary \ref{Athetaconcrete}, the algebra $A_\theta$ is generated by the following two operators: $Ue_{k,l} = e^{-2\pi il\theta}e_{k+1,l}$ and $Ve_{k,l} = e_{k,l+1}$, satisfying the relations: $VU = e^{2\pi i\theta}UV$, $U^*U =I$ and $V^*V=VV^* = I$.  We also reuse the following diagonal label operator $\K e_{k,l} = ke_{k,l}$, so, for a bounded function $f: \Z_{\ge0}\to\C$, we have $f(\K) e_{k,l} = f(k)e_{k,l}$.  We have the following two useful commutation relations for a diagonal operator $f(\K)$:

\begin{equation}\label{CommRel}
f(\K)\,U = Uf(\K +1)\textrm{ and }f(\K)\,V = Vf(\K).
\end{equation}
Let $Pol(U,V)$ be the set of all polynomials in $U$, $U^*$, $V$ and $V^*$.  We call a function $f:\Z_{\ge 0}\to\C$ {\it eventually constant}, if there exists a natural number $k_0$ such that $f(k)$ is constant for $k\ge k_0$. The set of all such functions will be denoted by $EC$. Using the above notation we consider the following operators in $\ell^2(\Z_{\geq 0}\times\Z)$ expressed as finite sums:

\begin{equation*}
a = \sum_{n=0,|m|\le M}^N V^mU^nf_{m,n}^+(\K) + \sum_{n=1,|m|\le M}^Nf_{m,n}^-(\K)V^m(U^*)^n
\end{equation*}
for some $M,N\geq 0$ and $f_{m,n}^\pm(k)\in EC$.
Let $\mathcal{A}$ be the set of all such operators.
We have the following proposition:

\begin{prop}
$\mathcal{A}= Pol(U,V)$.
\end{prop}

\begin{proof}
Using the commutation relations (\ref{CommRel}) it is easy to see that a product of two elements of $\mathcal{A}$ and the adjoint of an element of $\mathcal{A}$ are still in $\mathcal{A}$. It follows that $\mathcal{A}$ is a $*-$subalgebra of $A_\theta$.  Since $U$ and $V$ are in $\mathcal{A}$, it follows that $Pol(U,V)\subset\mathcal{A}$.  To prove the other way inclusion it suffices to show that for any $f\in EC$ the operator $f(\K)$ is in $Pol(U,V)$, as the remaining parts of the sum are already polynomials in $U$, $V$, $U^*$ and $V^*$.  To show that $f(\K)\in Pol(U,V)$,  we decompose any $f(\K)\in EC$ in the following way:

\begin{equation*}
f(\K) = \sum_{k=0}^{k_0-1}f(k)P_k + f_\infty P_{\ge k_0},
\end{equation*}
where $f_\infty = \lim_{k\to\infty}f(k)$, $P_k$ is the orthogonal projection onto $\textrm{span}\{e_{k,l}\}_{l\in\Z}$ and $P_{\ge k_0}$ is the orthogonal projection onto $\textrm{span}\{e_{k,l}\}_{k\ge k_0,l\in\Z}$.  A straightforward calculation shows that $U^k(U^*)^k = P_{\ge k}$ and $P_k = P_{\ge k} - P_{\ge k+1}$.  This completes the proof.
\end{proof}

The above considerations tell us that $a\in\mathcal{A}$ are completely determined by the coefficients $f_{m,n}^\pm(\K)\in EC$ and so it motivates the following definition of a partial Fourier series. For $f\in A_\theta$, define the formal series:

\begin{equation*}
f_{\textrm{series}} = \sum_{n\ge0,m\in\Z} V^mU^nf_{m,n}^+(\K) + \sum_{n\ge1,m\in\Z}f_{m,n}^-(\K)V^m(U^*)^n,
\end{equation*}
where

\begin{equation*}
f_{m,n}^+(k) = \langle e_{k,0}, (U^*)^nV^{-m}fe_{k,0}\rangle \quad\textrm{and}\quad f_{m,n}^-(k) = \langle e_{k,0}, fU^nV^{-m}e_{k,0}\rangle .
\end{equation*}

Similarly to the usual theory of Fourier series, $f_\textrm{series}$ determines $f$ even though in general the series is not norm convergent. Other types of convergence results can be obtained along the lines of the usual Fourier analysis.

We now proceed to describing the Hilbert spaces hosting our Dirac type operators. They are analogs of the $L^2$ spaces on the classical solid torus. Let $\{a_{n}(k)\}$ be a sequence of positive numbers, which we call weights, such that the sum
$$s(n):=\sum_{k=0}^\infty \frac{1}{a_{n}(k)}$$
exists and such that $s(n)$ goes to zero as $n\to\infty$.   
For any formal series define a norm by

\begin{equation*}
\|f_{\textrm{series}}\|^2 = \sum_{k=0}^\infty \sum_{n\geq0, m\in\Z} \frac{1}{a_{n}(k)}|f_{m,n}^+(k)|^2 + \sum_{k=0}^\infty \sum_{n\geq1, m\in\Z} \frac{1}{a_{n}(k)}|f_{m,n}^-(k)|^2.
\end{equation*}
Let $\mathcal{H}_0$ be the Hilbert space whose elements are the above formal series $f_{\textrm{series}}$ such that $\|f_{\textrm{series}}\|$ is finite.   

%
Let $Q_0$ be the orthogonal projection in $\ell^2(\Z_{\ge0}\times\Z)$  onto $\textrm{ span}\{e_{k,0}\}_{k\in\Z_{\ge0}}$.  Define the linear functional $\tau$ on $B_1(\ell^2(\Z_{\ge0}\times\Z))$, the space of trace class operators, by $\tau(a) = \textrm{tr}(aQ_0)$.

\begin{lem}\label{trace_est} 
For a bounded operator $a$, and $b\in B_1(\ell^2(\Z_{\ge0}\times\Z))$
we have the following inequality: 

\begin{equation*}
|\tau(ab)|\le \|a\|(\tau(b^*b))^{1/2} := \|a\|\|b\|_\tau.
\end{equation*}
\end{lem}

\begin{proof}
We estimate as follows:

\begin{equation*}
\begin{aligned}
|\tau(ab)| &= |\textrm{tr}(abQ_0)|\le \|a\|(\textrm{tr}((bQ_0)^*bQ_0))^{1/2} = \|a\|(\textrm{tr}(Q_0^*b^*bQ_0))^{1/2} \\
&= \|a\|(\textrm{tr}(b^*bQ_0))^{1/2} = \|a\|(\tau(b^*b))^{1/2},
\end{aligned}
\end{equation*}
where the inequality follows from the usual estimate for the trace.
\end{proof}

We use this lemma to show that $A_\theta$ is dense in $\mathcal{H}_0$.

\begin{prop}
If $f\in A_\theta$, then $f_\textrm{series}$ converges in $\mathcal{H}_0$.   Moreover the map  $A_\theta\ni f\mapsto f_\textrm{series}\in \mathcal{H}_0$ is continuous, one-to-one, and the image is dense in $\mathcal{H}_0$.
\end{prop}

\begin{proof}
We follow here the similar argument from \cite{KM3}. For $f\in A_\theta$ we need to estimate the norm of $f_\textrm{series}$. Notice first that if $f\in\mathcal{A}$, then $f_{\textrm{series}} = f$.  Since such $f$'s are dense in $A_\theta$, it suffices to estimate the norm of the finite sums.   Without loss of generality, we assume that $f$ has only the $U^*$ and $V$ terms as the same proof works for the remaining terms. For such $f$ we have:

\begin{equation*}
\begin{aligned}
&\|f_\textrm{series}\|_{\mathcal{H}_0}^2 = \sum_{k=0}^\infty \sum_{n=1,|m|\le M}^N \frac{1}{a_{n}(k)}|f_{m,n}^-(k)|^2 = \tau\left(\sum_{n=1,|m|\le M}^N \frac{1}{a_{n}(\K)}|f_{m,n}^-(\K)|^2\right) \\
&= \tau\left(\sum_{n=1,|m|\le M}^N \frac{1}{a_{n}(\K)}f_{m,n}^-(\K)V^{-m}\left(U^*\right)^nV^mU^n\overline{f_{m,n}^-(\K)}\right) \\
&=\tau\left(\left(\sum_{n=1,|m|\le M}^N \frac{1}{a_{n}(\K)}f_{m,n}^-(\K)V^m(U^*)^n\right)f^*\right). \\
\end{aligned}
\end{equation*}
Consequently, using the inequality from Lemma \ref{trace_est}, we get
\begin{equation*}
\|f_\textrm{series}\|_{\mathcal{H}_0}^2\leq
\left\|\sum_{n=1,|m|\le M}^N \frac{1}{a_{n}(\K)}f_{m,n}^-(\K)V^m(U^*)^n\right\|_\tau\|f\|.
\end{equation*}
Next we estimate:
\begin{equation*}
\begin{aligned}
&\left\|\sum_{n=1,|m|\le M}^N \frac{1}{a_n(\K)}f_{m,n}^-(\K)V^m(U^*)^n\right\|_\tau^2 = \\
&= \tau\left(\sum_{n=1,|m|\le M}^N\frac{1}{a_n(\K)}f_{m,n}^-(\K)V^m(U^*)^n\sum_{l=1,|j|\le M}^N V^{-j}U^l\overline{f_{j,l}^-(\K)}\frac{1}{a_n(\K)}\right) \\
&= \sum_{k=0}^\infty \sum_{n=1,|m|\le M}^N \frac{1}{a_n(k)a_n(k)}|f_{m,n}^-(k)|^2 \le \|f_{\textrm{series}}\|^2_{\mathcal{H}_0}\sup_{n,k}\left(\frac{1}{a_n(k)}\right) .
\end{aligned}
\end{equation*}
Using the  the summability conditions on the weights we obtain:

\begin{equation*}
\sup_{n,k}\left(\frac{1}{a_n(k)}\right) \le \sup_{n}\left(\sum_{k=0}^\infty \frac{1}{a_n(k)}\right)= \sup_{n} (s(n)) \le \textrm{const},
\end{equation*}
and hence we get $\|f_{\textrm{series}}\|_{\mathcal{H}_0}^2 \le \textrm{const}\|f_{\textrm{series}}\|_{\mathcal{H}_0}\|f\|$. This shows the continuity of the map $A_\theta\ni f\mapsto f_\textrm{series}\in \mathcal{H}_0$, and consequently $f_\textrm{series}$ converges in $\mathcal{H}_0$ for every $f\in A_\theta$.

Next we show that the map $A_\theta\ni f\mapsto f_\textrm{series}\in \mathcal{H}_0$ is one-to-one.   Let $f$ and $g$ belong to $A_\theta$ and suppose that $f_\textrm{series} = g_\textrm{series}$.   This means that the Fourier coefficients are equal, that is: $f_{m,n}^+(k) = g_{m,n}^+(k)$ for all $n\ge0$ and all $m$ and $k$, and $f_{m,n}^-(k) = g_{m,n}^-(k)$ for all $n\ge1$ and all $m$ and $k$.  From  $f_{m,n}^+(k) = \langle e_{k,0}, (U^*)^nV^{-m}fe_{k,0}\rangle$ and $f_{m,n}^-(k) = \langle e_{k,0}, fU^nV^{-m}e_{k,0}\rangle$ it follows that all matrix coefficients of $f$ and $g$ are the same so we must have $f=g$.   Thus the map $f\mapsto f_\textrm{series}\in  \mathcal{H}_0$ is one-to-one.

To prove density we define the following indicator function

\begin{equation*}
\delta_l(k) = \left\{
\begin{aligned}
&1 &&l=k \\
&0 &&l\neq k
\end{aligned}\right. .
\end{equation*}
 It is clear that $V^mU^n\delta_l(\K)$ and $\delta_l(\K)V^m(U^*)^n$ are in $A_\theta$, and moreover they form an orthogonal basis for  $\mathcal{H}_0$.  Finally  finite linear combinations of $V^mU^n\delta_l(\K)$ and $\delta_l(\K)V^m(U^*)^n$ form a dense subspace of $A_\theta$ as they are polynomials in $U$ and $V$ and hence are in $\mathcal{A}$, making it a dense subspace of $\mathcal{H}_0$.   Thus the proof is complete.
\end{proof}

\section{Dirac type operators}

The purpose of this section is to introduce the main object of study of this paper: Dirac type operators on the quantum solid torus. We start with reviewing the definition of such operators, contained in \cite{KM4}, on the classical solid torus. In that paper we considered the following formally self-adjoint Dirac operator $D_{cl}$ defined in $L^2(ST^2,\C^2) \cong L^2(ST^2)\otimes \C^2$ by

\begin{equation*}
D_{cl} = \left(
\begin{array}{cc}
\frac{1}{i}\frac{\d}{\d\theta} & 2\frac{\d}{\d\overline{z}} \\
-2\frac{\d}{\d z} & -\frac{1}{i}\frac{\d}{\d\theta}
\end{array}\right),
\end{equation*}
where $ST^2 = \D \times S^1$ is the solid torus, and $\D = \left\{z\in \C \ : \ |z|\leq 1\right\}$ is the unit disk, while $S^1 = \left\{e^{i\theta}\in \C \ : \ 0\leq\theta\leq 2\pi\right\}$ is the unit circle. Using Fourier decomposition:

\begin{equation*}
F = \sum_{m,n\in\Z} \left(
\begin{array}{c}
f_{m,n}(r) \\
g_{m,n}(r)
\end{array}\right) e^{in\f + im\theta},
\end{equation*}
the operator $D$ becomes:
\begin{equation*}
D_{cl}F=\sum_{m,n\in\Z}\left(
\begin{array}{cc}
m & e^{i\f}\left(\frac{\d}{\d r} - \frac{n}{r}\right) \\
e^{-i\f}\left(-\frac{\d}{\d r} + \frac{n}{r}\right) & -m
\end{array}\right) \left(
\begin{array}{c}
f_{m,n}(r) \\
g_{m,n}(r)
\end{array}\right) e^{in\f + im\theta}.
\end{equation*}

The boundary condition we studied in \cite{KM4}, geometrically very much the same as the APS condition, considered the solid torus as a subset of the bigger noncompact space of the plane cross the unit circle. We defined the domain of the Dirac operator $D_{cl}$ as consisting of those sufficiently regular (first Sobolev class) vectors $F$ which extend to  square integrable solutions of $D_{cl}F=0$ on the complement of the solid torus.

Denote by $I_n, K_n$  the modified Bessel functions of the first and second kind respectively. Without boundary conditions, the operator $D_{cl}$ on the solid torus has an infinite dimensional kernel consisting of linear combinations of special solutions $I_n$ for which, in Fourier decomposition, the only nonzero components for $m\ne 0$ are:
 \begin{equation*}
f_{m,n+1}(r)  = -\frac{m}{|m|}I_{n+1}(|m|r),\ \ g_{m,n}(r)  = I_n(|m|r).
\end{equation*}
Outside of the solid torus in the plane cross the unit circle the square integrable functions in kernel of $D_{cl}$ are linear combinations of solutions $K_n$ for which the only non vanishing components when $m\ne 0$ are:
 \begin{equation*}
f_{m,n+1}(r)  = \frac{m}{|m|}K_{n+1}(|m|r),\ \ g_{m,n}(r)  = K_n(|m|r).
\end{equation*}
Consequently, see \cite{KM4} for details, using the Fourier decomposition of $F$, the boundary condition can be rephrased as 
\begin{equation*}
|m|K_{n+1}(|m|)g_{m,n}(1) - mK_n(|m|)f_{m,n+1}(1)=0,
\end{equation*}
if $m\ne 0$ and $f_{0,n}(1)=0$
for $n\leq 0$. Additionally, for $m=0$, we have  $g_{0,n}(1)=0$ for $n\geq 0$.

We now proceed to define our Dirac type operators for the noncommutative solid torus, in a way that is analogous to the commutative case.
Let $c_{1,n}(k)\leq 1$ and $c_{2,n}(k)\leq 1$ be sequences of positive numbers such that $\prod_k c_{1,n}(k)$, $\prod_k c_{2,n}(k)$ exist and are not zero, and such that there exists a constant, $\kappa$, such that for all $k$ or $n$ we have $1/\kappa \le 1/c_{i,n}(k) \le 1$, $i=1,2$.   

Let $\ell_{a_n}^2(\Z_{\ge0})$ be the following Hilbert space of sequences:

\begin{equation*}
\ell_{a_n}^2(\Z_{\ge0}) = \left\{\{h(k)\}_{k\geq 0} \ : \ \sum_k \frac{1}{a_n(k)}|h(k)|^2 < \infty\right\}.
\end{equation*}
We introduce the Jacobi type difference operators by formulas:

\begin{equation}
\begin{aligned}
B_{n}h(k) &= a_{n}(k)(h(k) - c_{2,n}(k-1)h(k-1)) :\ell_{a_{n+1}}^2(\Z_{\ge0})\to \ell_{a_{n}}^2(\Z_{\ge0})\\
\overline{B}_{n}h(k) &= a_{n+1}(k)(h(k)-c_{1,n}(k)h(k+1)) :\ell_{a_{n}}^2(\Z_{\ge0})\to\ell_{a_{n+1}}^2(\Z_{\ge0})
\end{aligned}
\end{equation}
defined on maximal domains, i.e. $\textrm{dom}(B) = \{h\in \ell_{a_{n+1}}^2(\Z_{\ge0})\ : \ \|Bh\|_{\ell_{a_{n}}^2(\Z_{\ge0})} < \infty\}$ and the domain of $\overline{B}$ is defined in the same way.   

Using the above general one-step difference operators we define operators $\delta_0$ and $\delta_2$ which are noncommutative analogs of 
$\frac{\d}{\d{z}}$ and $\frac{\d}{\d\overline{z}}$ as:
\begin{equation*}
\begin{aligned}
\delta_0(f) &= -\sum_{n\geq 0, m\in\Z} V^mU^{n+1}\overline{B}_{n}f_{m,n}^+(\K) + \sum_{n\geq 1, m\in\Z} B_{n-1}f_{m,n}^-(\K)V^m\left(U^*\right)^{n-1} \\
\delta_2(f) &= -\sum_{n\geq 1, m\in\Z} V^mU^{n-1}B_{n-1}f_{m,n}^+(\K) + \sum_{n\geq 0,m\in \Z} \overline{B}_{n}f_{m,n}^-(\K)V^m\left(U^*\right)^{n+1}.
\end{aligned}
\end{equation*}
Those operators are more general versions of the d-bar like operators considered in \cite{KM1}.

Next define another diagonal label operator $\mathbb{L}$ on $\ell^2(\Z_{\ge0}\times\Z)$ by $\mathbb{L} e_{k,l} = le_{k,l}$. We use it for the definition of $\delta_1 = [\mathbb{L},\ \cdot\ ]$, the analog of $\frac{1}{i}\frac{\d}{\d\theta}$.   It is easy to see that

\begin{equation*}
\delta_1(f) = \sum_{n\geq 0, m\in\Z} mV^mU^nf_{m,n}^+(\K) + \sum_{n\geq1, m\in\Z} mf_{m,n}^-(\K)V^m\left(U^*\right)^n,
\end{equation*}
again considered on its maximal domain.

Let $\mathcal{H} = \mathcal{H}_0 \otimes \C^2$.   Any $F\in\mathcal{H}$ has the following Fourier decomposition:

\begin{equation}\label{F_Fourier_decomp}
F = \left(
\begin{array}{c}
\sum_{n\ge1,m\in\Z}V^mU^nf_{m,n}^+(\K) + \sum_{n\ge0,m\in\Z}f_{m,n}^-(\K)V^m(U^*)^n \\
\sum_{n\ge0,m\in\Z}V^mU^ng_{m,n}^+(\K) + \sum_{n\ge1,m\in\Z}g_{m,n}^-(\K)V^m(U^*)^n
\end{array}\right).
\end{equation}
  
We define the quantum Dirac type operator $D$ to be

\begin{equation}\label{D_Def}
D = \left(
\begin{array}{cc}
\delta_1 & \delta_0 \\
\delta_2 & -\delta_1
\end{array}\right) .
\end{equation}
Initially we set the domain of $D$ to be the maximal domain:
\begin{equation}
\textrm{dom}(D) = \{F\in\mathcal{H} \ : \ \|DF\|< \infty \}.
\end{equation}

Using Fourier series (\ref{F_Fourier_decomp}) we relate $D$ to finite difference operators with matrix coefficients.  For brevity throughout the rest of the paper we only study the ``positive'' part of the Fourier transform of the elements of the Hilbert space, as the ``negative'' part can be analyzed in an analogical fashion, but writing down all the repetative details would have significantly increase the length of the paper (and decrease its readability).   Because of this, the ``$+$'' superscript is dropped for simplicity.

%
%

%

Define the following Jacobi type difference operator with matrix valued coefficients

\begin{equation}\label{matrix_finite_diff_oper}
\mathcal{A}_{m,n}\left(
\begin{array}{c}
f \\
g
\end{array}\right)(k) = A_{m,n}(k+1)\left[\left(
\begin{array}{c}
f(k+1) \\ g(k+1)
\end{array}\right) - C_{m,n}(k)\left(
\begin{array}{c}
f(k) \\ g(k)
\end{array}\right)\right],
\end{equation}
where

\begin{equation}\label{A_matrix}
A_{m,n}(k+1) = \left(
\begin{array}{cc}
a_{n+1}(k)c_{1,n}(k) & 0 \\
m & a_{n}(k+1)
\end{array}\right),
\end{equation}
and

\begin{equation}\label{C_matrix}
C_{m,n}(k) = \left(
\begin{array}{cc}
\frac{1}{c_{1,n}(k)} & \frac{-m}{a_{n+1}(k)c_{1,n}(k)} \\
\frac{-m}{a_{n}(k+1)c_{1,n}(k)} & c_{2,n}(k) + \frac{m^2}{a_{n}(k+1)a_{n+1}(k)c_{1,n}(k)}
\end{array}\right).
\end{equation}
Notice that $\mathcal{A}_{m,n}: \ell_{a_n}^2(\Z_{\ge0})\otimes\ell_{a_{n+1}}^2(\Z_{\ge0}) \to \ell_{a_{n+1}}^2(\Z_{\ge0})\otimes\ell_{a_n}^2(\Z_{\ge0})$.  Initially we define the domain for $\mathcal{A}_{m,n}$ to be

\begin{equation}
\textrm{dom }\mathcal{A}_{m,n} = \left\{h\in\ell_{a_n}^2(\Z_{\ge0})\otimes\ell_{a_{n+1}}^2(\Z_{\ge0})\ : \|\mathcal{A}_{m,n}h\|_{a_{n+1}\otimes a_n} < \infty\right\}.
\end{equation}
The domain will be modified later to include a boundary condition.

\begin{prop}\label{Db_parametrix_rel}
For $F,G\in\mathcal{H}$ where $F = (f,g)^t$, $G=(p,q)^t$, and $f,g,p,q\in\mathcal{H}_0$, the ``positive'' part of the equation $DF = G$ is equivalent to the following equations:
\begin{equation}
\begin{aligned}
\mathcal{A}_{m,n} \left(
\begin{array}{c}
g_{m,n} \\ f_{m,n+1}
\end{array}\right) (k)= \left(
\begin{array}{c}
p_{m,n+1}(k) \\ -q_{m,n}(k+1)
\end{array}\right),
\end{aligned}
\end{equation}
for $n\geq 0$, $m\in\Z$, $k\geq 0$, and 
\begin{equation}\label{DFcomp0}
a_n(0)f_{m,n+1}(0) + mg_{m,n}(0)=q_{m,n}(0).
\end{equation}
\end{prop}
The last equation (\ref{DFcomp0}) will be referred to below as the initial regularity condition as it is both a starting point of recurrence proofs and it is analogous to the regularity condition at origin for solutions of Dirac operator for the classical solid torus.

\begin{proof}
In the Fourier series for $F$ and $G$ there is an ambiguity whether $n=0$ term should go with positive or with negative terms. Because we change summation indices in the calculation below, the ambiguity is resolved differently for different components of $F$ and $G$.   

Using the definition of $D$ and shifting $n$ in the first sum we get, ignoring the ``negative'' part,

\begin{equation*}
DF = \sum_{n\ge0,m\in\Z} V^m \left(
\begin{array}{c}
U^{n+1}(mf_{m,n+1}(k) - \overline{B}_{n}g_{m,n}(k)) \\
U^n(-B_{n}f_{m,n+1}(k) - mg_{m,n}(k))
\end{array}\right) = G .
\end{equation*}
Shifting $k$ in the second equation for $k\geq 1$ we get the following equivalent system of equations

\begin{equation*}
\left\{
\begin{aligned}
mf_{m,n+1}(k) - a_{n+1}(k)(g_{m,n}(k)-c_{1,n}(k)g_{m,n}(k+1)) &= p_{m,n+1}(k) \\
a_{n}(k+1)(f_{m,n+1}(k+1) - c_{2,n}(k)f_{m,n+1}(k)) + mg_{m,n}(k+1) & =-q_{m,n}(k+1),
\end{aligned}\right.
\end{equation*}
while the case $k=0$ in the second equation leads to equation (\ref{DFcomp0}).   

The system above can be rewritten in matrix form:
\begin{equation*}
\begin{aligned}
&\left(
\begin{array}{cc}
a_{n+1}(k)c_{1,n}(k) & 0 \\
m & a_{n}(k+1)
\end{array}\right)\left(
\begin{array}{c}
g_{m,n}(k+1) \\ f_{m,n+1}(k+1)
\end{array}\right) \\
&\qquad - \left(
\begin{array}{cc}
a_{n+1}(k) & -m \\
0 & a_{n}(k+1)c_{2,n}(k)
\end{array}\right)\left(
\begin{array}{c}
g_{m,n}(k) \\ f_{m,n+1}(k)
\end{array}\right) = \left(
\begin{array}{c}
p_{m,n+1}(k) \\ -q_{m,n}(k+1)
\end{array}\right).
\end{aligned}
\end{equation*}
The first matrix above is $A_{m,n}(k+1)$, therefore factoring it out of the left side of the equation gives the desired result and completes the proof.
\end{proof}

\section{Properties of the coefficients}

It follows from Proposition \ref{Db_parametrix_rel} that to analyze $D$ we need to study the properties of the operators $\mathcal{A}_{m,n}$ subject to the initial regularity condition. This section discusses important properties of the matrix coefficients of those operators.

Unless specified differently, in all formulas in the section, $n\geq 0$, $m\in\Z$, $k\geq 0$. First notice that $\det A_{m,n}(k+1) = a_{n+1}(k)a_{n}(k+1)c_{1,n}(k) \neq 0$ for any $k$ and $n$ which means that the inverse of the matrices $A_{m,n}(k+1)$ exists for any $k$ and $n$.   
Notice also that $\det C_{m,n}(k) = c_{2,n}(k)/c_{1,n}(k)$. Additionally we have:

\begin{prop}\label{C_existence}
The infinite product

\begin{equation*}
C_{m,n} := \prod_{k=0}^\infty C_{m,n}(k)
\end{equation*}
exists and is invertible.
\end{prop}

Here and everywhere below a product of matrices is understood from right to left, that is:

\begin{equation*}
\prod_{i=0}^k C_{m,n}(i) = C_{m,n}(k)C_{m,n}(k-1)\cdots C_{m,n}(0).
\end{equation*}

\begin{proof}
We consider $C_{m,n}(k) - I$, since by \cite{Trgo}, if the series $\sum_k \|C_{m,n}(k) - I\|$ converges and $\textrm{det }C_{m,n}(k)\neq 0$, then the infinite product of $C_{m,n}(k)$ converges.   We already have $\textrm{det } C_{m,n}(k)\neq 0$ for all $k,m,n$, so consider

\begin{equation*}
C_{m,n}(k) - I = \left(
\begin{array}{cc}
\frac{1}{c_{1,n}(k) - 1} & -\frac{m}{a_{n+1}(k)c_{1,n}(k)} \\
-\frac{m}{a_n(k+1)c_{1,n}(k)} & c_{2,n}(k) - 1 + \frac{m^2}{a_n(k+1)a_{n+1}(k)c_{1,n}(k)}
\end{array}\right).
\end{equation*}
Estimating the matrix norm by the sum of absolute values of the coefficients, we have

\begin{equation*}
\begin{aligned}
\|C_{m,n}(k) - I\| &\le \left|\frac{1}{c_{1,n}(k) - 1}\right| + \frac{|m|}{a_{n+1}(k)c_{1,n}(k)} \\
&\quad\quad+ \frac{|m|}{a_n(k+1)c_{1,n}(k)} + \left|c_{2,n}(k) - 1\right| + \frac{m^2}{a_n(k+1)a_{n+1}(k)c_{1,n}(k)}.
\end{aligned}
\end{equation*}
By assumption, the infinite products of $c_{1,n}(k)$ and $c_{2,n}(k)$ exist. Using this and the summability criteria on $1/a_n(k)$, we see that $\sum_k \|C_{m,n}(k) - I\|$ converges.   Therefore we can deduce that the infinite product $\prod_k C_{m,n}(k)$ exists, thus completing proof.
\end{proof}

Below we use simplified notation for the products:

\begin{equation*}
J_1(n) = \prod_{k=0}^\infty c_{1,n}(k)\quad\textrm{and}\quad J_2(n) = \prod_{k=0}^\infty c_{2,n}(k).
\end{equation*}
We have

\begin{equation*}
\det C_{m,n} = \lim_{k\to\infty}\prod_{i=0}^k \det C_{m,n}(i) = \lim_{k\to\infty}\prod_{i=0}^k\frac{c_{2,n}(i)}{c_{1,n}(i)} = \frac{\prod_{i=0}^\infty c_{2,n}(i)}{\prod_{i=0}^\infty c_{1,n}(i)} = \frac{J_2(n)}{J_1(n)}.
\end{equation*}

The next proposition describes the structure of the infinite product of the $C_{m,n}(k)$ matrices.

\begin{prop}\label{C_structure}
The infinite product $C_{m,n}$ has the following structure:
\begin{equation*}
C_{m,n} = \left(
\begin{array}{cc}
\prod_{i=0}^\infty\frac{1}{c_{1,n}(i)} + F_0(m^2) & -mF_1(m^2) \\
-mF_2(m^2) & \prod_{i=0}^\infty c_{2,n}(i) + F_3(m^2)
\end{array}\right),
\end{equation*}
where the $F_j(m^2)$ are power series in $m^2$ with positive coefficients for $j=0,1,2,3$, and additionally

\begin{equation*}
F_3(m^2) = \left(\prod_{i=0}^\infty \frac{1}{c_{1,n}(i)} + F_0(m^2)\right)^{-1}\left(m^2F_1(m^2)F_2(m^2) - F_0(m^2)\prod_{i=0}^\infty c_{2,n}(i)\right) .
\end{equation*}
\end{prop}

\begin{proof}
It is easy to verify by induction that for each $k$ the product $\prod_{i=0}^k C_{m,n}(i)$ is of the form

\begin{equation*}
\prod_{i=0}^k C_{m,n}(i) = \left(
\begin{array}{cc}
\prod_{i=0}^k \frac{1}{c_{1,n}(i)} + \sum_{i=0}^k u_0(n,k)(m^2)^i & -m\sum_{i=0}^k u_1(n,k)(m^2)^i \\
-m\sum_{i=0}^k u_2(n,k)(m^2)^i & \prod_{i=0}^k c_{2,n}(i) + \sum_{i=0}^k u_3(n,k)(m^2)^i 
\end{array}\right),
\end{equation*}
where the above sums are polynomials in $m^2$ with positive coefficients.   Moreover as $k\to\infty$, the polynomials converge to a power series, by the Weierstrass Analytic Convergence Theorem, see \cite{Ahlfors}.   Thus the result follows.
\end{proof}

\section{The Class of Boundary Conditions}

As defined above, the operator $D$ is not a Fredholm operator on its maximal domain: it has an infinite dimensional kernel, as is typical for differential operators on manifolds with boundary. In this section we single out two types of solutions  in the kernel of $\mathcal{A}_{m,n}$  that play the role of similar solutions in \cite{KM4} made of modified Bessel functions of first and second kind. Like in the commutative case, they are then used to define a class boundary conditions for $D$ that turns $D$ into an invertible operator. First we describe the kernel of $D$ through the kernel of $\mathcal{A}_{m,n}$.

 
\begin{prop}
Let $\mathcal{A}_{m,n}$ be the operator given by equation (\ref{matrix_finite_diff_oper}), then

\begin{equation*}
\textrm{Ker }\mathcal{A}_{m,n} = \left\{\left(\prod_{i=0}^{k-1} C_{m,n}(i)\right)v \right\}
\end{equation*}
for some vector $v$.
\end{prop}

\begin{proof}
By formula (\ref{matrix_finite_diff_oper}) we need to solve the equation:

\begin{equation*}
A_{m,n}(k+1)\left[\left(
\begin{array}{c}
f(k+1) \\ g(k+1)
\end{array}\right) - C_{m,n}(k)\left(
\begin{array}{c}
f(k) \\ g(k)
\end{array}\right)\right] = \left(
\begin{array}{c}
0 \\ 0
\end{array}\right).
\end{equation*}
Solving recursively we see that

\begin{equation*}
\left(
\begin{array}{c}
f(k+1) \\ g(k+1)
\end{array}\right) = \left(\prod_{i=0}^k C_{m,n}(i)\right)v
\end{equation*}
for some arbitrary vector $v$.
\end{proof}

If $F_{m,n}(k)\in \textrm{Ker }\mathcal{A}_{m,n}$, then by Proposition \ref{C_existence} $\lim_{k\to\infty} F_{m,n}(k) = F_{m,n}(\infty)$ exists and is finite.

We define $I_{m,n}(k)$ to be a special element of $\textrm{Ker }\mathcal{A}_{m,n}$, so:

\begin{equation}\label{nc_special_soln_I}
I_{m,n}(k) = \left(\prod_{i=0}^{k-1} C_{m,n}(i)\right)I_{m,n}(0) = \left(\prod_{i=0}^{k-1} C_{m,n}(i)\right)\left(
\begin{array}{c}
I_{m,n}^{(1)}(0) \\ I_{m,n}^{(2)}(0)
\end{array}\right),
\end{equation}
which additionally satisfies the initial regularity equation (\ref{DFcomp0}) with zero right-hand side.  We normalize $I_{m,n}(k)$ in the following way: we set

\begin{equation*}
I_{m,n}^{(1)}(0) = -1 \quad\textrm{and}\quad I_{m,n}^{(2)}(0) = \frac{m}{a_{n}(0)}.
\end{equation*}

We also want to define a class of complementary solutions, denoted $K_{m,n}(k)$, for which

\begin{equation}\label{nc_special_soln_K}
K_{m,n}(k) = \left(\prod_{i=0}^{k-1} C_{m,n}(i)\right)K_{m,n}(0) = \left(\prod_{i=0}^{k-1} C_{m,n}(i)\right) \left(
\begin{array}{c}
K_{m,n}^{(1)}(0) \\ K_{m,n}^{(2)}(0)
\end{array}\right),
\end{equation}
and such that for every $k,n\ge0$ and $m\in\Z$, the set $\{I_{m,n}(k), K_{m,n}(k)\}$ is linearly independent.   This can be achieved by requesting the following first three sign conditions at infinity: 

\begin{equation}\label{nc_K_conditions}
\begin{aligned}
&1.)\ m>0:\ K_{m,n}^{(1)}(\infty)>0 \textrm{ and } K_{m,n}^{(2)}(\infty)>0 \\
&2.)\ m<0:\ K_{m,n}^{(1)}(\infty)<0 \textrm{ and } K_{m,n}^{(2)}(\infty)>0 \\
&3.)\ m=0:\ K_{0,n}^{(1)}(\infty) = 0 \textrm{ and } K_{0,n}^{(2)}(\infty)\ne 0\\
&4.)\ m\neq 0:\ \frac{K_{m,n}^{(1)}(\infty)}{K_{m,n}^{(2)}(\infty)}\to 0\quad\textrm{as}\quad |m|\to\infty\textrm{ uniformly in }n.
\end{aligned}
\end{equation}
Additionally, the sign conditions, and the normalization of $I_{m,n}$, lead to monotonicity properties of $I_{m,n}$ and $K_{m,n}$ that are crucially used in estimates in the last section.
The key asymptotic fourth condition above is also required to prove compactness of the resolvent. Any collection $K_{m,n}(k)$ satisfying (\ref{nc_special_soln_K}) and (\ref{nc_K_conditions}) will be referred to, abusing the terminology, as a quantum $K$ function.

We have:

\begin{prop}\label{linear_independ}
Let $K_{m,n}(k)$ be any quantum $K$ function. For every $k,n\ge0$ and $m\in\Z$ the vectors $I_{m,n}(k)$ and $K_{m,n}(k)$ are linearly independent  in $\C^2$. Moreover for all $n\ge0$ and $m\in\Z$, $I_{m,n}(\infty)$ and $K_{m,n}(\infty)$ are also linearly independent.
\end{prop}

\begin{proof}

Consider the case $m>0$.  Recall that we have

\begin{equation*}
I_{m,n}(k) = \left(\prod_{i=0}^{k-1} C_{m,n}(i)\right)\left(
\begin{array}{c}
-1 \\ \frac{m}{a_{n}(0)}
\end{array}\right)\quad\textrm{and}\quad K_{m,n}(k) = \left(\prod_{i=0}^{k-1} C_{m,n}(i)\right)\left(
\begin{array}{c}
K_{m,n}^{(1)}(0) \\ K_{m,n}^{(2)}(0)
\end{array}\right)  .
\end{equation*}
Using the formulas for the product of $C_{m,n}(k)$ in the proof of Proposition \ref{C_structure} we write out the components to get 

\begin{equation*}
\begin{aligned}
&I_{m,n}^{(1)}(k) = -\left(\prod_{i=0}^k \frac{1}{c_{1,n}(i)} + \sum_{i=0}^k u_0(n,k)(m^2)^i  + \frac{m^2\sum_{i=0}^k u_1(n,k)(m^2)^i }{a_{n}(0)}\right), \\
&I_{m,n}^{(2)}(k) = m\sum_{i=0}^k u_2(n,k)(m^2)^i + \frac{m}{a_{n}(0)}\left(\prod_{i=0}^k c_{2,n}(i) + \sum_{i=0}^k u_3(n,k)(m^2)^i \right).
\end{aligned}
\end{equation*}
Since all coefficients are positive and $m>0$, we see that $I_{m,n}^{(1)}(k)$ is negative and $I_{m,n}^{(2)}(k)$ is positive.  

To analyze $K_{m,n}(k)$ we use an alternative, equivalent, way to write it as:

\begin{equation*}
K_{m,n}(k) = \left(\prod_{i=k}^\infty C_{m,n}(i)\right)^{-1}K_{m,n}(\infty) .
\end{equation*}
Since the components of $K_{m,n}(\infty)$ are both positive, and the matrix $C_{m,n}(i)^{-1}$ has all positive entries for $m>0$, it follows that $K_{m,n}^{(1)}(k)$ and $K_{m,n}^{(2)}(k)$ are positive.   Therefore since one of the components of $I_{m,n}(k)$ is negative and both components of $K_{m,n}(k)$ are positive, it is impossible for them to be linearly dependent.   

In the case $m<0$ both components of $I_{m,n}(k)$ are negative and one of the components of $K_{m,n}(k)$ is positive again showing they can not be linearly dependent.   When $m=0$ then all matrices $C_{m,n}(k)$ are diagonal and so $K_{m,n}^{(1)}(k)=0$ and $I_{m,n}^{(2)}(k)=0$, which implies independence. Finally the same arguments are valid when $k=\infty$, thus the proof is complete.
\end{proof}

Now we can define the class of boundary conditions we consider for $D$.   Given an $F\in\mathcal{H}$ such that $DF\in\mathcal{H}$, the Fourier coefficients $f_{m,n+1}(k)$ and $g_{m,n}(k)$ given by equation (\ref{F_Fourier_decomp}), have limits as $k\to\infty$ for all $n\ge0$ and $m\in\Z$.   This follows from the calculation of the resolvent in the next section, see equation (\ref{Q_at_infinity}) for details.   These limits will be denoted by $f_{m,n+1}(\infty)$ and $g_{m,n}(\infty)$.   

Any quantum $K$ function defines a boundary condition in the following way.
%

\begin{defin}
Given $K_{m,n}(k)$ satisfying (\ref{nc_special_soln_K}) and (\ref{nc_K_conditions}) we set the domain dom$(D)$ of $D$
to be the space of all $F\in\mathcal{H}$ such that $DF\in\mathcal{H}$ and such that for all $m\in\Z$, $n\ge0$ there exist $\beta_{m,n}\in\C$ such that
\begin{equation}\label{bndy_cond}
\begin{aligned}
g_{m,n}(\infty) &= \beta_{m,n}K_{m,n}^{(1)}(\infty) \\
 f_{m,n+1}(\infty) &= \beta_{m,n}K_{m,n}^{(2)}(\infty).
\end{aligned}
\end{equation}
\end{defin}

The above conditions at infinity can be restated to mimic the boundary condition in the classical case:

\begin{equation*}
g_{m,n}(\infty)K_{m,n}^{(2)}(\infty) - f_{m,n+1}(\infty)K_{m,n}^{(1)}(\infty) = 0.
\end{equation*}

The following theorem is the main result of this paper and is proved at the end of this paper.

\begin{theo}\label{quant_Q_theorem}
The quantum Dirac operator $D$, defined by (\ref{D_Def}) subject to the boundary condition defined in equation (\ref{bndy_cond}), is an invertible operator whose inverse $Q$, is a compact operator.
\end{theo}

We end this section with listing of useful recurrence relations for $I_{m,n}(k)$ and $K_{m,n}(k)$ that are easy consequences of their definitions. Let $H_{m,n}(k)$ be $I_{m,n}(k)$ or $K_{m,n}(k)$ with components $H_{m,n}(k) = (H_{m,n}^{(1)}(k), H_{m,n}^{(2)}(k))^t$, then we have

\begin{equation}\label{recurrence_matrix}
H_{m,n}(k+1) = C_{m,n}(k)H_{m,n}(k).
\end{equation}
This gives the following recurrence relation:

\begin{equation}\label{recurrence_1}
\begin{aligned}
&H_{m,n}^{(1)}(k+1) - \frac{1}{c_{1,n}(k)}H_{m,n}^{(1)}(k) = -\frac{m}{a_{n+1}(k)c_{2,n}(k)}H_{m,n}^{(2)}(k) \\
&H_{m,n}^{(2)}(k+1) - c_{2,n}(k)H_{m,n}^{(2)}(k) =- \frac{m}{a_{n}(k+1)c_{1,n}(k)}H_{m,n}^{(1)}(k) \\
&+ \frac{m^2}{a_{n+1}(k)a_{n}(k+1)c_{1,n}(k)}H_{m,n}^{(2)}(k).
\end{aligned}
\end{equation}
Similarly, using the relation

\begin{equation*}
H_{m,n}(k) = (C_{m,n}(k))^{-1}H_{m,n}(k+1),
\end{equation*}
we can produce the following two (equivalent) equations:

\begin{equation}\label{recurrence_2}
\begin{aligned}
&H_{m,n}^{(2)}(k) - \frac{1}{c_{2,n}(k)}H_{m,n}^{(2)}(k+1) = \frac{m}{a_{n}(k+1)c_{2,n}(k)}H_{m,n}^{(1)}(k+1) \\
&H_{m,n}^{(1)}(k) - c_{1,n}(k)H_{m,n}^{(1)}(k+1) = \frac{m}{a_{n+1}(k)c_{2,n}(k)}H_{m,n}^{(2)}(k+1) +\\
& + \frac{m^2}{a_{n+1}(k)a_{n}(k+1)c_{2,n}(k)}H_{m,n}^{(1)}(k+1) .
\end{aligned}
\end{equation}
The recurrence relations above are extensively used in the next two section in the analysis of the parametrix of $D$.

\section{Parametrix of the quantum Dirac type operator}
In the previous section we introduced a new domain of $D$.  To account for this, we redefine the domain of $\mathcal{A}_{m,n}$ in the following way:
dom$(\mathcal{A}_{m,n})$ is the space of all $h\in\ell_{a_n}^2(\Z_{\ge0})\otimes\ell_{a_{n+1}}^2(\Z_{\ge0})$ such that $\|\mathcal{A}_{m,n}h\|_{a_{n+1}\otimes a_n} < \infty$, and such that there exists $\beta\in\C$ such that
\begin{equation}\label{new_domain_matrix_diff_oper}
\begin{aligned}
x(\infty) &= \beta K_{m,n}^{(1)}(\infty) \\
y(\infty) &= \beta K_{m,n}^{(2)}(\infty),
\end{aligned}
\end{equation}
where $h = (x,y)^t$.  The existence of such limits $x(\infty)$ and $y(\infty)$ follows from the calculation of the resolvent in this section, see equation (\ref{Q_at_infinity}).

\begin{prop}
Given $h\in\textrm{dom }\mathcal{A}_{m,n}$, if $h\in\textrm{Ker }\mathcal{A}_{m,n}$ and $h = (x,y)^t$ satisfies 

\begin{equation*}
a_n(0)y(0) + mx(0) = 0,
\end{equation*}
then $h(k)=0$ for every $k\ge0$. In other words, the kernel of $\mathcal{A}_{m,n}$ subject to initial regularity condition (\ref{DFcomp0}) is trivial.
\end{prop}

\begin{proof}
First notice that we can write any element $h(k)\in \textrm{Ker }\mathcal{A}_{m,n}$ as 

\begin{equation*}
h(k) = c_1I_{m,n}(k) + c_2K_{m,n}(k),
\end{equation*}
since by Proposition \ref{linear_independ}, the nonzero two-vectors $I_{m,n}(k)$ and $K_{m,n}(k)$ are linearly independent.   However if $k=0$ then both $h(0)$ and $I_{m,n}(0)$ satisfy the equation $a_n(0)y(0) + mx(0) = 0$ so there exists some constant $c_3$ such that:

\begin{equation*}
c_3I_{m,n}(0) = h(0) = c_1I_{m,n}(0) + c_2K_{m,n}(0).
\end{equation*}
This would imply that $K_{m,n}(0)$ is a scalar multiple of $I_{m,n}(0)$, which is impossible by the linear independence of the solutions, hence $c_2=0$.   If $k=\infty$ then, since $h\in\textrm{dom }\mathcal{A}_{m,n}$, there is some constant $c_4$, such that:

\begin{equation*}
c_4K_{m,n}(\infty) = h(\infty) = c_1I_{m,n}(\infty),
\end{equation*}
which implies that $K_{m,n}(\infty)$ is a scalar multiple of $I_{m,n}(\infty)$, but again this is impossible by Proposition \ref{linear_independ}, hence $c_1=0$.   Therefore $h(k)=0$ for every $k$.   
%
This completes the proof.
\end{proof}

Next we discuss the non-homogeneuous equation $DF=G$, the solution of which leads to the parametrix (in our case the inverse) of the quantum Dirac type operator $D$, subject to boundary conditions (\ref{bndy_cond}).  

We use the following notation: for a vector $v=(v_1,v_2)^t$ we write $v^\perp := (v_2,-v_1)^t$.

\begin{prop}
Let $\mathcal{A}_{m,n}$ be the one-step matrix difference operator defined by equation (\ref{matrix_finite_diff_oper}), then $\mathcal{A}_{m,n}$, subject to the boundary conditions given by the equation (\ref{new_domain_matrix_diff_oper}), and subject to the initial regularity condition (\ref{DFcomp0}) is an invertible operator with inverse $Q^{(m,n)}$ given by (\ref{Q_def}) below.
\end{prop}

\begin{proof}

The goal is to solve the equation

\begin{equation*}
\mathcal{A}_{m,n}\left(
\begin{array}{c}
x \\ y
\end{array}\right)(k) = \left(
\begin{array}{c}
p(k) \\ -q(k+1)
\end{array}\right),
\end{equation*}
which becomes the following difference equation with matrix coefficients

\begin{equation*}
A_{m,n}(k+1)\left[\left(
\begin{array}{c}
x(k+1) \\ y(k+1)
\end{array}\right) - C_{m,n}(k)\left(
\begin{array}{c}
x(k) \\ y(k)
\end{array}\right)\right] = \left(
\begin{array}{c}
p(k) \\ -q(k+1)
\end{array}\right)
\end{equation*}
with $A_{m,n}(k+1)$ and $C_{m,n}(k)$ defined in formulas (\ref{A_matrix}) and (\ref{C_matrix}) respectively, while  the initial regularity condition is
\begin{equation*}
a_n(0)y(0)+mx(0)=q(0).
\end{equation*}
 Relabeling $h(k) = (x(k), y(k))^t$ and $r(k+1) = (p(k), -q(k+1))^t$, the system becomes 
$$A_{m,n}(k+1)(h(k+1) - C_{m,n}(k)h(k)) = r(k+1).$$   
Before applying the boundary condition, the solution $h(k)$ is not unique, with non-uniqueness due to one-dimensional kernel $I_{m,n}(k)$.
With this in mind we choose $h(0)$ to be the following particular solution of the initial regularity condition:
\begin{equation*}
h(0)=\left(0,\frac{q(0)}{a_n(0)}\right)^t.
\end{equation*}

%
%
Solving the difference equation by using variation of constants method we get

\begin{equation}\label{initial_Q}
Q^{(m,n)}r(k) := h(k) = \prod_{i=0}^{k-1}C_{m,n}(i)\sum_{i=0}^k\left(\prod_{j=0}^{i-1} C_{m,n}(j)\right)^{-1}\left(A_{m,n}(i)\right)^{-1}r(i) + \alpha I_{m,n}(k)
\end{equation}
for some parameter $\alpha$.  Here, for convenience, we set the $\prod_{j=k}^{k-1} C_{m,n}(j) =1$ for any $m$ and $n$, and also we define 
$\left(A_{m,n}(0)\right)^{-1}r(0) :=h(0) = \left(0,\frac{q(0)}{a_n(0)}\right)^t$ from above.
To apply the boundary conditions in equation (\ref{bndy_cond}), we need to know that $Q^{(m,n)}r(\infty)$ is well defined; this follows from Proposition \ref{C_existence} and the summability of $A_{m,n}(k)^{-1}$.   Therefore the boundary condition, equation (\ref{bndy_cond}), is well defined, and gives:

\begin{equation}\label{Q_at_infinity}
\begin{aligned}
Q^{(m,n)}r(\infty) &= C_{m,n}\sum_{i=0}^\infty \left(\prod_{j=0}^{i-1} C_{m,n}(j)\right)^{-1}\left(A_{m,n}(i)\right)^{-1}r(i) + \alpha I_{m,n}(\infty)=\\
 &= \beta K_{m,n}(\infty)
\end{aligned}
\end{equation}
for  some  constant $\beta$.  
The goal is to solve for $\alpha$. Below we use the following notation:
\begin{equation}
\tau_{m,n}:= \langle K_{m,n}(0), I_{m,n}(0)^\perp\rangle.
\end{equation}

Multiplying by $(C_{m,n})^{-1}$ and taking the inner product of both sides of equation (\ref{Q_at_infinity}) with $K_{m,n}(0)^\perp$, we get

\begin{equation*}
\left< \sum_{i=0}^\infty \left(\prod_{j=0}^{i-1} C_{m,n}(j)\right)^{-1}\left(A_{m,n}(i)\right)^{-1}r(i), K_{m,n}(0)^\perp\right> + \alpha\tau_{m,n} = 0,
\end{equation*}
which can now be solved for $\alpha$:

\begin{equation*}
\alpha = \frac{-1}{\tau_{m,n}}\left< \sum_{i=0}^\infty \left(\prod_{j=0}^{i-1} C_{m,n}(j)\right)^{-1}\left(A_{m,n}(i)\right)^{-1}r(i), K_{m,n}(0)^\perp\right> .
\end{equation*}

There is a partial cancellation between the two terms in equation (\ref{initial_Q}). To see this, we notice that,
since $I_{m,n}(0)$ and $K_{m,n}(0)$ are linearly independent, we can decompose any two-vector $v$ as $$v = v_1I_{m,n}(0) + v_2K_{m,n}(0),$$ where

\begin{equation*}
v_1 = -\frac{\left<v, K_{m,n}(0)^\perp\right>}{\tau_{m,n}} \qquad\textrm{and}\qquad v_2 = \frac{\left<v, I_{m,n}(0)^\perp\right>}
{\tau_{m,n}} ,
\end{equation*}

Applying this decomposition to vector

\begin{equation*}
v = \sum_{i=0}^k\left(\prod_{j=0}^{i-1} C_{m,n}(j)\right)^{-1}\left(A_{m,n}(i)\right)^{-1}r(i),
\end{equation*}
and using the formula for $\alpha$, and the formulas for $I_{m,n}(k)$ and $K_{m,n}(k)$, equations (\ref{nc_special_soln_I}) and (\ref{nc_special_soln_K}) respectively, we get

\begin{equation*}
\begin{aligned}
&Q^{(m,n)}r(k) = \prod_{i=0}^{k-1} C_{m,n}(i) \\
&\times\left[-\frac{1}{\tau_{m,n}}\left<\sum_{i=k+1}^\infty \left(\prod_{j=0}^{i-1} C_{m,n}(j)\right)^{-1}\!\!\!\!\left(A_{m,n}(i)\right)^{-1}r(i), K_{m,n}(0)^\perp\right>I_{m,n}(0) \right. \\
&\left. +\frac{1}{\tau_{m,n}}\left<\sum_{i=0}^k\left(\prod_{j=0}^{i-1} C_{m,n}(j)\right)^{-1}\left(A_{m,n}(i)\right)^{-1}r(i), I_{m,n}(0)^\perp\right>K_{m,n}(0)\right]. \\
\end{aligned}
\end{equation*}
The coefficients in the formula above will be denoted by:

\begin{equation}\label{non_hom_c_{1,n}}
e_{m,n}^{(1)}(k) := \frac{-1}{\tau_{m,n}}\left<\sum_{i=k+1}^\infty \left(\prod_{j=0}^{i-1} C_{m,n}(j)\right)^{-1}\left(A_{m,n}(i)\right)^{-1}r(i), K_{m,n}(0)^\perp\right>
\end{equation}
and

\begin{equation}\label{non_hom_c_{2,n}}
e_{m,n}^{(2)}(k) := \frac{1}{\tau_{m,n}}\left<\sum_{i=0}^k\left(\prod_{j=0}^{i-1} C_{m,n}(j)\right)^{-1}\left(A_{m,n}(i)\right)^{-1}r(i), I_{m,n}(0)^\perp\right> .
\end{equation}
With this notation we get the formula for $Q^{(m,n)}$:

\begin{equation}\label{Q_def}
\begin{aligned}
Q^{(m,n)}r(k) &= e_{m,n}^{(1)}(k)I_{m,n}(k) + e_{m,n}^{(2)}(k)K_{m,n}(k) && m\neq 0 \\
Q^{(m,n)}r(k) &= \sum_{i=0}^k \left(
\begin{array}{cc}
0 & 0 \\
0 & \frac{1}{a_{n}(i)}\prod_{j=i}^{k-1} c_{2,n}(j)
\end{array}\right)r(i)&& m=0.
\end{aligned}
\end{equation}
The proof is complete.
\end{proof}

As computed above, the operator $Q^{(m,n)}$ is not easy to analyze, mostly because it contains products of non-commuting matrices $C_{m,n}(k)$. Our strategy is to re-write the formulas for $Q^{(m,n)}$ in terms of quantum $I$ and $K$ functions, similarly to the commutative case, and then estimate the Hilbert-Schmidt norm of the parametrics using the recurrence relations (\ref{recurrence_1}) and (\ref{recurrence_2}).

We need the following observations.  For vectors $v=(v_1,v_2)^t$ and $u=(u_1,u_2)^t$ we have $\langle v^\perp, u \rangle = -\langle v, u^\perp\rangle$, where $\langle \cdot, \cdot\rangle$ is the standard Euclidean inner product on $\R^2$.   
 If $R$ is a $2\times2$ matrix and $v=(v_1,v_2)^t$, then we have $(Rv)^\perp = (\textrm{det }R)\left(R^t\right)^{-1}v^\perp$.  
  
As a consequence of the above and using the formula for the determinant of the $C_{m,n}(k)$ matrix, we get:

\begin{equation}\label{special_soln_perp}
\begin{aligned}
I_{m,n}(k)^\perp &= \prod_{i=0}^{k-1}\frac{c_{2,n}(i)}{c_{1,n}(i)}\left[\left(\prod_{i=0}^{k-1} C_{m,n}(i)\right)^{-1}\right]^t I_{m,n}(0)^\perp \\
K_{m,n}(k)^\perp &= \prod_{i=0}^{k-1}\frac{c_{2,n}(i)}{c_{1,n}(i)}\left[\left(\prod_{i=0}^{k-1} C_{m,n}(i)\right)^{-1}\right]^t K_{m,n}(0)^\perp .
\end{aligned}
\end{equation}
Consequently, inserting equations (\ref{special_soln_perp}) into formulas for the coefficients $e_{m,n}^{(1)}$, $e_{m,n}^{(2)}$, (\ref{non_hom_c_{1,n}}) and (\ref{non_hom_c_{2,n}}), gives:

\begin{equation*}
\begin{aligned}
&e_{m,n}^{(1)}(k) = 
\frac{1}{\tau_{m,n}}\sum_{i=k+1}^\infty \prod_{j=0}^{i-1}\frac{c_{1,n}(j)}{c_{2,n}(j)}(K_{m,n}^{(2)}(i), -K_{m,n}^{(1)}(i))\left(A_{m,n}(i)\right)^{-1}r(i) .
\end{aligned}
\end{equation*}
Similarly we also have:

\begin{equation*}
e_{m,n}^{(2)}(k) = \frac{1}{\tau_{m,n}}\sum_{i=0}^k \prod_{j=0}^{i-1}\frac{c_{1,n}(j)}{c_{2,n}(j)}(I_{m,n}^{(2)}(i), -I_{m,n}^{(1)}(i))\left(A_{m,n}(i)\right)^{-1}r(i).
\end{equation*}
These are to be understood as a row of a matrix times a column vector.   

It is convenient to rewrite $Q^{(m,n)}$ in the matrix notation:

\begin{equation*}
\begin{aligned}
Q^{(m,n)}r_n(k) &= e_{m,n}^{(1)}(k)I_{m,n}(k) + e_{m,n}^{(2)}(k)K_{m,n}(k) \\
&= (I_{m,n}(k), K_{m,n}(k)) \left(
\begin{array}{c}
e_{m,n}^{(1)}(k) \\ e_{m,n}^{(2)}(k)
\end{array}\right) = \left(
\begin{array}{cc}
I_{m,n}^{(1)}(k) & K_{m,n}^{(1)}(k) \\
I_{m,n}^{(2)}(k) & K_{m,n}^{(2)}(k)
\end{array}\right)\left(
\begin{array}{c}
e_{m,n}^{(1)}(k) \\ e_{m,n}^{(2)}(k)
\end{array}\right).
\end{aligned}
\end{equation*}

Inserting the above formulas for the coefficients $e_{m,n}^{(1)}$, $e_{m,n}^{(2)}$ we get:

\begin{equation*}
\begin{aligned}
&Q^{(m,n)}r_n(k) = \frac{1}{\tau_{m,n}} \left(
\begin{array}{cc}
I_{m,n}^{(1)}(k) & K_{m,n}^{(1)}(k) \\
I_{m,n}^{(2)}(k) & K_{m,n}^{(2)}(k)
\end{array}\right)\times\\
&\left(
\begin{array}{cc}
\sum_{i=k+1}^\infty \prod_{j=0}^{i-1}\frac{c_{1,n}(j)}{c_{2,n}(j)}K_{m,n}^{(2)}(i) & -\sum_{i=k+1}^\infty \prod_{j=0}^{i-1}\frac{c_{1,n}(j)}{c_{2,n}(j)}K_{m,n}^{(1)}(i) \\
\sum_{i=0}^k \prod_{j=0}^{i-1}\frac{c_{1,n}(j)}{c_{2,n}(j)}I_{m,n}^{(2)}(i) & -\sum_{i=0}^k \prod_{j=0}^{i-1}\frac{c_{1,n}(j)}{c_{2,n}(j)}I_{m,n}^{(1)}(i)
\end{array}\right) \left(A_{m,n}(i)\right)^{-1}r(i).
\end{aligned}
\end{equation*}
Next we focus on multiplying the second matrix with $\left(A_{m,n}(i)\right)^{-1}$.   We use the recurrence relations, (\ref{recurrence_1}) and (\ref{recurrence_2}), to get: 

%

\begin{equation}\label{Q_through_KandI}
\begin{aligned}
&Q^{(m,n)}r_n(k) = \frac{1}{\tau_{m,n}} \left(
\begin{array}{cc}
I_{m,n}^{(1)}(k) & K_{m,n}^{(1)}(k) \\
I_{m,n}^{(2)}(k) & K_{m,n}^{(2)}(k)
\end{array}\right)\times\\
&\left(
\begin{array}{cc}
\sum_{i=k+1}^\infty \prod_{j=0}^{i-2}\frac{c_{1,n}(j)}{c_{2,n}(j)}\frac{1}{a_{n+1}(i-1)}K_{m,n}^{(2)}(i-1) & \sum_{i=k+1}^\infty \prod_{j=0}^{i-1}\frac{c_{1,n}(j)}{c_{2,n}(j)}\frac{1}{a_{n}(i)}K_{m,n}^{(1)}(i) \\
\sum_{i=0}^k \prod_{j=0}^{i-2}\frac{c_{1,n}(j)}{c_{2,n}(j)}\frac{1}{a_{n+1}(i-1)}I_{m,n}^{(2)}(i-1) & \sum_{i=0}^k \prod_{j=0}^{i-1}\frac{c_{1,n}(j)}{c_{2,n}(j)}\frac{1}{a_{n}(i)}I_{m,n}^{(1)}(i)
\end{array}\right)r(i).
\end{aligned}
\end{equation}

This formula finally expresses $Q^{(m,n)}$ through the quantum $I$ and $K$ functions. Its primary benefit is that it shows that the entries the matrix of  $Q^{(m,n)}$
are integral operators in weighted $\ell^2$ spaces. Below we explicitly write down those operators.

For $\alpha,\beta =1,2$, $m\neq0$ and $n\ge0$, we define the following integral operators  $X_{m,n}^{\alpha\beta} , Y_{m,n}^{\alpha\beta} : \ell_{a_{n-1+\beta}}^2(\Z_{\ge0}) \to \ell_{a_{n-1+\alpha}}^2(\Z_{\ge0})$ by:

\begin{equation}\label{int_ops_m_not_zero}
\begin{aligned}
&X_{m,n}^{\alpha\beta}r(k) = I_{m,n}^{(\alpha)}(k) \sum_{i=k+1}^\infty \left(\prod_{j=0}^{i-\beta}\frac{c_{1,n}(j)}{c_{2,n}(j)}\right)\frac{K_{m,n}^{(\beta)}(i-\beta+1)}{a_{n-1+\beta}(i-\beta+1)}r(i) \\
&Y_{m,n}^{\alpha\beta}r(k) = K_{m,n}^{(\alpha)}(k) \sum_{i=0}^k \left(\prod_{j=0}^{i-\beta}\frac{c_{1,n}(j)}{c_{2,n}(j)}\right)\frac{I_{m,n}^{(\beta)}(i-\beta+1)}{a_{n-1+\beta}(i-\beta+1)}r(i).
\end{aligned}
\end{equation}
Also, for $n\ge0$, define integral operator $Z_{0,n}:\ell_{a_{n}}^2(\Z_{\ge0}) \to \ell_{a_{n+1}}^2(\Z_{\ge0})$
by

\begin{equation}\label{int_ops_m_zero}
Z_{0,n} r(k) = \sum_{i=0}^k \left(\prod_{j=i}^{k-1}c_{2,n}(j)\right) \frac{r(i)}{a_{n}(i)}.
\end{equation}

\begin{lem}\label{para_equiv_form}
The parametrix, $Q^{(m,n)}$ for the operator $\mathcal{A}_{m,n}$ from above for $m\neq 0$ and $n\ge0$ is given by the following equivalent formula:

\begin{equation*}
Q^{(m,n)}\left(
\begin{array}{c}
r^{(1)}(k) \\ r^{(2)}(k)
\end{array}\right) = \frac{1}{\tau_{m,n}}\left(
\begin{array}{c}
p_{m,n}^{(1)}(k) \\ p_{m,n}^{(2)}(k)
\end{array}\right),
\end{equation*}
where

\begin{equation*}
\begin{aligned}
&p_{m,n}^{(1)}(k) = X_{m,n}^{12}r^{(1)}(k) + Y_{m,n}^{12}r^{(1)}(k) + X_{m,n}^{11}r^{(2)}(k) + Y_{m,n}^{11}r^{(2)}(k) \\
&p_{m,n}^{(2)}(k) = X_{m,n}^{22}r^{(1)}(k) + Y_{m,n}^{22}r^{(1)}(k) + X_{m,n}^{21}r^{(2)}(k) + Y_{m,n}^{21}r^{(2)}(k).
\end{aligned}
\end{equation*}
Moreover when $m=0$ and $n\ge0$ the parametrix is given by the following formula:

\begin{equation*}
Q^{(0,n)}\left(
\begin{array}{c}
r^{(1)}(k) \\ r^{(2)}(k)
\end{array}\right) = \left(
\begin{array}{c}
0 \\ Z_{0,n}r^{(2)}(i)
\end{array}\right).
\end{equation*}
\end{lem}

\begin{proof}
Multiplying out the two matrices in (\ref{Q_through_KandI}) and applying it to the vector $r(i) = \left(r^{(1)}(i),r^{(2)}(i)\right)$, and using the definitions of the integral operators, equation (\ref{int_ops_m_not_zero}), gives the desired result.   The case $m=0$ immediately follows from equations (\ref{Q_def}) and (\ref{int_ops_m_zero}).   Thus the proof is complete.
\end{proof}

\section{Analysis of the Parametrix}
This section contains analysis of the parametrix $Q^{(m,n)}$, culminating in the proof of the main theorem \ref{quant_Q_theorem}. 

The case $m=0$ is the easiest: all matrices are diagonal and the analysis is not any different than the analysis for the quantum disk \cite{KM1} and \cite{KM3}. 
Below we only consider the case $m>0$ since for $m<0$ the computations are virtually identical.   We start by gathering the basic information about the quantum $I$ and $K$ functions.

\begin{lem}
For $m>0$ the components $-I_{m,n}^{(1)}(k)$, $I_{m,n}^{(2)}(k)$, $K_{m,n}^{(1)}(k)$, and $K_{m,n}^{(2)}(k)$ are all positive for all $n,k\ge0$.
\end{lem}

\begin{proof}
This immediately follows from the definitions and the proof of Proposition \ref{linear_independ}.
\end{proof}

\begin{lem}\label{special_soln_ineq_1}
We have the following inequalities for all $n,k\ge0$ and $m>0$:

\begin{equation*}
\begin{aligned}
&-I_{m,n}^{(1)}(k) < -I_{m,n}^{(1)}(k+1), && I_{m,n}^{(2)}(k) < \frac{1}{c_{2,n}(k)}I_{m,n}^{(2)}(k+1), \\
&K_{m,n}^{(1)}(k+1) < \frac{1}{c_{1,n}(k)}K_{m,n}^{(1)}(k), && K_{m,n}^{(2)}(k+1) < K_{m,n}^{(2)}(k).
\end{aligned}
\end{equation*}
\end{lem}

\begin{proof} 
Using recurrence relation (\ref{recurrence_1}), (\ref{recurrence_2}), the assumption that $m>0$ and the previous lemma we have

\begin{equation*}
0< \frac{m}{a_{n}(k+1)}K_{m,n}^{(1)}(k+1) = c_{2,n}(k)K_{m,n}^{(2)}(k) - K_{m,n}^{(2)}(k+1).
\end{equation*}
Consequently, using $c_{2,n}(k)\le 1$ , we have

\begin{equation*}
K_{m,n}^{(2)}(k+1) < c_{2,n}(k)K_{m,n}^{(2)}(k) \le K_{m,n}^{(2)}(k),
\end{equation*}
proving one of the inequalities.   Similarly, we have

\begin{equation*}
0< \frac{m}{a_{n+1}(k)c(k)}K_{m,n}^{(2)}(k) = \frac{1}{c_{1,n}(k)}K_{m,n}^{(1)}(k) - K_{m,n}^{(1)}(k+1),
\end{equation*}
implying the other inequality involving $K_{m,n}$.   The proofs for the $I_{m,n}(k)$ are identical. 
\end{proof}

The next lemma is the crux of the compactness argument. It establishes estimates on the components of  quantum $I$ and $K$ functions in a similar fashion to the estimates of the modified Bessel functions in  \cite{KM4}.   Consider the quantity: 

\begin{equation*}
\varepsilon(m,n) = \sum_{k=0}^\infty \frac{a_{n+1}(k)}{m^2 + a_{n}(k)a_{n+1}(k)}.
\end{equation*}
The series above is convergent by summability assumptions on $a_n(k)$.

\begin{lem}\label{special_soln_ineq_2}
For $m\neq0$ and $n,k\ge0$ the following inequalities hold:

\begin{equation*}
I_{m,n}^{(2)}(k) \le -|m|\varepsilon(m,n)I_{m,n}^{(1)}(k+1)\quad\textrm{and}\quad K_{m,n}^{(1)}(k+1) \le |m|\left(\varepsilon(m,n) + \frac{K_{m,n}^{(1)}(\infty)}{K_{m,n}^{(2)}(\infty)}\right)K_{m,n}^{(2)}(k).
\end{equation*}
Moreover $\varepsilon(m,n)\to 0$ as $|m|,n\to\infty$.
\end{lem}

\begin{proof}
Assuming $m>0$ and using the recurrence relations (\ref{recurrence_1}), (\ref{recurrence_2}), we have

\begin{equation*}
\begin{aligned}
\frac{I_{m,n}^{(2)}(k)}{-I_{m,n}^{(1)}(k+1)} &= \frac{c_{1,n}(k)}{\frac{m}{a_{n+1}(k)} + \frac{1}{\frac{m}{a_{n}(k)} + c_{2,n}(k-1)\frac{I_{m,n}^{(2)}(k-1)}{-I_{m,n}^{(1)}(k)}}} \\
&\le c_{1,n}(k)\left(\frac{\frac{m}{a_{n}(k)}}{1+\frac{m^2}{a_{n}(k)a_{n+1}(k)}} + c_{2,n}(k-1)\frac{I_{m,n}^{(2)}(k-1)}{-I_{m,n}^{(1)}(k)}\right) \\
&\le \frac{\frac{m}{a_{n}(k)}}{1+\frac{m^2}{a_{n}(k)a_{n+1}(k)}} + \frac{I_{m,n}^{(2)}(k-1)}{-I_{m,n}^{(1)}(k)},
\end{aligned}
\end{equation*}
since  $c_{1,n}(k)\le 1$ and $c_{2,n}(k)\le 1$ for all $k\geq 0$, assuming $I_{m,n}^{(2)}(-1)=0$.   It follows, by rearranging the terms, that

\begin{equation*}
\frac{1}{m}\left(\frac{I_{m,n}^{(2)}(k)}{-I_{m,n}^{(1)}(k+1)} - \frac{I_{m,n}^{(2)}(k-1)}{-I_{m,n}^{(1)}(k)}\right) \le \frac{\frac{1}{a_{n}(k)}}{1+ \frac{m^2}{a_{n}(k)a_{n+1}(k)}}.
\end{equation*}
Summing both sides and telescoping the left side we get

\begin{equation*}
\frac{1}{m}\left(\frac{I_{m,n}^{(2)}(k)}{-I_{m,n}^{(1)}(k+1)}\right) \le \sum_{k=0}^\infty \frac{\frac{1}{a_{n}(k)}}{1+ \frac{m^2}{a_{n}(k)a_{n+1}(k)}} = \varepsilon(m,n) .
\end{equation*}
This establishes the first inequality.   To obtain the second inequality we proceed similarly, using the recurrence relations (\ref{recurrence_1}), (\ref{recurrence_2}), but this time telescoping sum goes to infinity.

The next step is to show $\varepsilon(m,n)$ goes to zero as $m$ and $n$ go to infinity.  It follows immediately from the definition that:

\begin{equation*}
\varepsilon(m,n) \le \sum_{k=0}^\infty \frac{1}{a_{n}(k)},
\end{equation*}
and the sum on the right goes to zero as $n\to\infty$ from the condition on $a_{n}(k)$, and hence $\varepsilon(m,n)\to 0$ as $n\to\infty$.   Now for any $\eta>0$ pick $N>0$ such that we have

\begin{equation*}
\sum_{k>N} \frac{1}{a_{n}(k)} \le \frac{\eta}{2},
\end{equation*}
and pick $M>0$ such that

\begin{equation*}
\sum_{k\le N} \frac{\frac{1}{a_{n}(k)}}{1 + \frac{m^2}{a_{n}(k)a_{n+1}(k)}} \le \frac{\eta}{2}
\end{equation*}
for $m>M$.   It now follows that $\varepsilon(m,n)\to 0$ as $n,m\to\infty$.   The case $m<0$  is analogous, and thus the desired result follows.
\end{proof}

The next technical lemma deals with sums appearing in the definitions of the integral operators which comprise the parametrix.

\begin{lem}\label{sum_trick}
The following summation estimates are true for $m\neq0$ and $n,k\ge0$:

\begin{equation*}
\sum_{i=k+1}^\infty \left(\prod_{j=0}^{i-2}c_{1,n}(j)\right)\frac{K_{m,n}^{(2)}(i-1)}{a_{n+1}(i-1)} \le \frac{1}{|m|}\prod_{j=0}^{k-1} c_{1,n}(j)K_{m,n}^{(1)}(k)\ \textrm{and}\ \sum_{i=k+1}^\infty \frac{K_{m,n}^{(1)}(i)}{a_{n}(i)} \le \frac{1}{|m|}K_{m,n}^{(2)}(k).
\end{equation*}
\end{lem}

\begin{proof}
As usual we consider only $m>0$.    
Using equation (\ref{recurrence_1}), multiplying both sides of the equation by a product of $c_{1,n}(j)$'s, and summing we get:

\begin{equation*}
\begin{aligned}
\sum_{i=k+1}^\infty \left(\prod_{j=0}^{i-2}c_{1,n}(j)\right)\frac{K_{m,n}^{(2)}(i-1)}{a_{n+1}(i-1)} &=\frac{1}{m}\sum_{i=k+1}^\infty \left(\prod_{j=0}^{i-2} c_{1,n}(j)K_{m,n}^{(1)}(i-1) - \prod_{j=0}^{i-1}c_{1,n}(j)K_{m,n}^{(1)}(i)\right) \\
&\le \frac{1}{m}\prod_{j=0}^{k-1}c_{1,n}(j)K_{m,n}^{(1)}(k),
\end{aligned}
\end{equation*}
where the last inequality is true because the difference above is a telescoping sum. The second inequality follows from the recurrence relation  (\ref{recurrence_2}) and the same telescoping sum trick.   This completes the proof.
\end{proof}

The next and final lemma in preparation for the proof of the main result deals with pointwise estimates of the products of different components of quantum $I$ and $K$ functions through the quantity $\tau_{m,n}= \langle K_{m,n}(0), I_{m,n}(0)^\perp\rangle$.

\begin{lem}\label{special_sol_bnd}
 If $m\neq0$ then, for all $n\ge 0$,

\begin{equation*}
\begin{aligned}
K_{m,n}^{(1)}(k)I_{m,n}^{(2)}(k) \leq \tau_{m,n} \prod_{i=0}^{k-1}\frac{c_{2,n}(i)}{c_{1,n}(i)} &\quad\textrm{and}\quad -K_{m,n}^{(2)}(k)I_{m,n}^{(1)}(k) \leq \tau_{m,n} \prod_{i=0}^{k-1}\frac{c_{2,n}(i)}{c_{1,n}(i)} \\
-I_{m,n}^{(1)}(k+1)K_{m,n}^{(2)}(k) \le \frac{\tau_{m,n}}{c_{1,n}(k)}\prod_{i=0}^{k-1}\frac{c_{2,n}(i)}{c_{1,n}(i)} &\quad\textrm{and}\quad I_{m,n}^{(2)}(k)K_{m,n}^{(1)}(k+1)\le \frac{\tau_{m,n}}{c_{1,n}(k)}\prod_{i=0}^{k-1}\frac{c_{2,n}(i)}{c_{1,n}(i)} .
\end{aligned}
\end{equation*}
\end{lem}

\begin{proof}
Assume as usual that $m>0$.   Using the matrix valued recurrence relation (\ref{recurrence_matrix}) and properties of the inner product we get

\begin{equation}\label{Wron_like}
\begin{aligned}
\langle K_{m,n}(k+1), I_{m,n}(k+1)^\perp\rangle &= \prod_{i=0}^k \frac{c_{2,n}(i)}{c_{1,n}(i)}\left< \prod_{i=0}^k C_{m,n}(i) K_{m,n}(0), \prod_{i=0}^k\left(C_{m,n}(i)^{-1}\right)^tI_{m,n}(0)^\perp\right> \\
&=\prod_{i=0}^k \frac{c_{2,n}(i)}{c_{1,n}(i)}\langle K_{m,n}(0), I_{m,n}(0)^\perp \rangle =\prod_{i=0}^k \frac{c_{2,n}(i)}{c_{1,n}(i)}\tau_{m,n} .
\end{aligned}
\end{equation}
Writing out the above inner product gives

\begin{equation*}
\langle K_{m,n}(k+1), I_{m,n}(k+1)^\perp\rangle = K_{m,n}^{(1)}(k+1)I_{m,n}^{(2)}(k+1) + K_{m,n}^{(2)}(k+1)[-I_{m,n}^{(1)}(k+1)].
\end{equation*}
Consequently, the above equality (\ref{Wron_like}) implies that 

\begin{equation*}
K_{m,n}^{(1)}(k)I_{m,n}^{(2)}(k) \leq \tau_{m,n}\prod_{i=0}^{k-1}\frac{c_{2,n}(i)}{c_{1,n}(i)} \quad\textrm{and}\quad -K_{m,n}^{(2)}(k)I_{m,n}^{(1)}(k) \leq \tau_{m,n}\prod_{i=0}^{k-1}\frac{c_{2,n}(i)}{c_{1,n}(i)} .
\end{equation*}

%
%
To get the second set of inequalities we use the above equation (\ref{Wron_like}) and recurrence relations (\ref{recurrence_1}), (\ref{recurrence_2}) to get

\begin{equation*}
\begin{aligned}
\tau_{m,n}\prod_{i=0}^{k-1}\frac{c_{2,n}(i)}{c_{1,n}(i)} &= -I_{m,n}^{(1)}(k)K_{m,n}^{(2)}(k) + I_{m,n}^{(2)}(k)K_{m,n}^{(1)}(k) \\
&= \left(-c_{1,n}(k)I_{m,n}^{(1)}(k+1)-\frac{m}{a_{n+1}(k)}I_{m,n}^{(2)}(k)\right)K_{m,n}^{(2)}(k) \\
&+ I_{m,n}^{(2)}(k)\left(c_{1,n}(k)K_{m,n}^{(1)}(k+1)+\frac{m}{a_{n+1}(k)}K_{m,n}^{(2)}(k)\right) \\
&=c_{1,n}(k)\left(-I_{m,n}^{(1)}(k+1)K_{m,n}^{(2)}(k) + I_{m,n}^{(2)}(k)K_{m,n}^{(1)}(k+1)\right) .
\end{aligned}
\end{equation*}
From the last equality it follows that: 

\begin{equation*}
I_{m,n}^{(1)}(k+1)K_{m,n}^{(2)}(k) \le \frac{\tau_{m,n}}{c_{1,n}(k)}\prod_{i=0}^{k-1}\frac{c_{2,n}(i)}{c_{1,n}(i)} \quad\textrm{and}\quad I_{m,n}^{(2)}(k)K_{m,n}^{(1)}(k+1)\le \frac{\tau_{m,n}}{c_{1,n}(k)}\prod_{i=0}^{k-1}\frac{c_{2,n}(i)}{c_{1,n}(i)},
\end{equation*}
and the lemma is proved.
\end{proof}

We have gathered now enough information to analyze compactness of the inverse of $D$. The Fourier transform decomposes that inverse into a direct sum of parametrices $Q_{m,n}$, which in turn consist of the integral operators $X_{m,n}^{\alpha\beta}$, $Y_{m,n}^{\alpha\beta}$, $Z_{0,n}$. Those integral operators are compact operators, in fact even Hilbert-Schmidt operators as stated in the following theorem. 

\begin{theo}
Let $n\ge0$.   If $m\neq0$, then the integral operators $X_{m,n}^{\alpha\beta}$ and $Y_{m,n}^{\alpha\beta}$ defined in equation (\ref{int_ops_m_not_zero}) are Hilbert-Schmidt operators for $\alpha,\beta =1,2$, and, if $m=0$, the integral operator $Z_{0,n}$ defined in equation (\ref{int_ops_m_zero}) is a Hilbert-Schmidt operator. Moreover the Hilbert-Schmidt norms of $X_{m,n}^{\alpha\beta}$ and $Y_{m,n}^{\alpha\beta}$ go to zero as $|m|,n\to\infty$, and the Hilbert-Schmidt norm of $Z^{(0,n)}$ goes to zero as $n\to\infty$.
\end{theo}

\begin{proof}
We start with the case $m\neq0$.   Using the definition of $X_{m,n}^{\alpha\beta}$ and $Y_{m,n}^{\alpha\beta}$, it is easy to see that for $\alpha,\beta =1,2$ we have

\begin{equation}
\begin{aligned}
& \|X_{m,n}^{\alpha\beta}\|_{HS}^2 = \sum_{k=0}^\infty \frac{\left(I_{m,n}^{(\alpha)}(k)\right)^2}{a_{n-1+\alpha}(k)}\sum_{i=k+1}^\infty \left(\prod_{j=0}^{i-\beta}\frac{c_{1,n}(j)}{c_{2,n}(j)}\right)^2\cdot \frac{\left(K_{m,n}^{(\beta)}(i-\beta+1)\right)^2}{a_{n-1+\beta}(i-\beta+1)} \\
& \|Y_{m,n}^{\alpha\beta}\|_{HS}^2 = \sum_{k=0}^\infty \frac{\left(K_{m,n}^{(\alpha)}(k)\right)^2}{a_{n-1+\alpha}(k)}\sum_{i=0}^k \left(\prod_{j=0}^{i-\beta}\frac{c_{1,n}(j)}{c_{2,n}(j)}\right)^2\cdot \frac{\left(I_{m,n}^{(\beta)}(i-\beta+1)\right)^2}{a_{n-1+\beta}(i-\beta+1)} .
\end{aligned}
\end{equation}
There are eight norms to estimate; however that number can be reduced to four, since by Fubini's Theorem we have:

\begin{equation*}
\begin{aligned}
&\|X_{m,n}^{11}\|_{HS}^2 = \|Y_{m,n}^{11}\|_{HS}^2, \quad  \|X_{m,n}^{22}\|_{HS}^2 = \|Y_{m,n}^{22}\|_{HS}^2 \\
&\|X_{m,n}^{12}\|_{HS}^2 = \|Y_{m,n}^{21}\|_{HS}^2, \quad  \|X_{m,n}^{21}\|_{HS}^2 = \|Y_{m,n}^{12}\|_{HS}^2.
\end{aligned}
\end{equation*}
As usual consider the case $m>0$.  Using Lemmas \ref{special_soln_ineq_1}, \ref{sum_trick}, \ref{special_sol_bnd} in the formula for the Hilbert-Schmidt norm of $X_{m,n}^{11}$, we obtain:

\begin{equation*}
\begin{aligned}
\|X_{m,n}^{11}\|_{HS}^2 &= \sum_{k=0}^\infty \frac{\left(I_{m,n}^{(1)}(k)\right)^2}{a_{n}(k)}\sum_{i=k+1}^\infty \left(\prod_{j=0}^{i-1}\frac{c_{1,n}(j)}{c_{2,n}(j)}\right)^2\cdot\frac{\left(K_{m,n}^{(1)}(i)\right)^2}{a_{n}(i)} \\
&\le \frac{J_1(n)^2}{J_2(n)^2}\kappa\sum_{k=0}^\infty \frac{\left(I_{m,n}^{(1)}(k)\right)^2K_{m,n}^{(1)}(k+1)}{a_{n}(k)}\sum_{i=k+1}^\infty \frac{K_{m,n}^{(1)}(i)}{a_{n}(i)} \\
&\le \frac{\tau_{m,n}\kappa}{m}\cdot\frac{J_1(n)}{J_2(n)}\sum_{k=0}^\infty \frac{\left[-I_{m,n}^{(1)}(k)\right]K_{m,n}^{(1)}(k+1)}{a_{n}(k)} .
\end{aligned}
\end{equation*}
We estimate the last term in the following way:

\begin{equation*}
\begin{aligned}
&\frac{\tau_{m,n}\kappa J_1(n)}{mJ_2(n)}\sum_{k=0}^\infty \frac{\left[-I_{m,n}^{(1)}(k)\right]K_{m,n}^{(1)}(k+1)}{a_{n}(k)}\cdot\frac{K_{m,n}^{(2)}(k)}{K_{m,n}^{(2)}(k)}\\
 &\le \frac{J_2(n)}{J_1(n)}\frac{\tau_{m,n}^2\kappa J_1(n)}{J_2(n)}\sum_{k=0}^\infty \left(\varepsilon(m,n) + \frac{K_{m,n}^{(1)}(\infty)}{K_{m,n}^{(2)}(\infty)}\right)\frac{1}{a_{n}(k)} \\
&= \tau_{m,n}^2\kappa\left(\varepsilon(m,n) + \frac{K_{m,n}^{(1)}(\infty)}{K_{m,n}^{(2)}(\infty)}\right) s(n),
\end{aligned}
\end{equation*}
by using Lemmas \ref{special_soln_ineq_2}, \ref{special_sol_bnd}, and the inequality $1/c_{1,n}(k)\le\kappa$.   Therefore we have

\begin{equation*}
\|X_{m,n}^{11}\|_{HS}^2 \le \tau_{m,n}^2\kappa\left(\varepsilon(m,n) + \frac{K_{m,n}^{(1)}(\infty)}{K_{m,n}^{(2)}(\infty)}\right)s(n).
\end{equation*}

Similarly, using Lemmas \ref{special_soln_ineq_1}, \ref{sum_trick}, and \ref{special_sol_bnd}, we can estimate the Hilbert-Schmidt norm of $X_{m,n}^{22}$ by:

\begin{equation*}
\begin{aligned}
\|X_{m,n}^{22}\|_{HS}^2 &= \sum_{k=0}^\infty \frac{\left(I_{m,n}^{(2)}(k)\right)^2}{a_{n+1}(k)}\sum_{i=k+1}^\infty\left(\prod_{j=0}^{i-2}\frac{c_{1,n}(j)}{c_{2,n}(j)}\right)^2\frac{\left(K_{m,n}^{(2)}(i-1)\right)^2}{a_{n+1}(i-1)} \\
&\le \frac{J_1(n)}{J_2(n)^2}\sum_{k=0}^\infty \frac{\left(I_{m,n}^{(2)}(k)\right)^2K_{m,n}^{(2)}(k)}{a_{n+1}(k)}\sum_{i=k+1}^\infty \left(\prod_{j=0}^{i-2}c_{1,n}(j)\right)\frac{K_{m,n}^{(2)}(i-1)}{a_{n+1}(i-1)} \\
&\le \frac{J_1(n)^2}{J_2(n)^2}\frac{\tau_{m,n}}{m}\cdot\frac{J_2(n)}{J_1(n)}\sum_{k=0}^\infty \frac{I_{m,n}^{(2)}(k)K_{m,n}^{(2)}(k)}{a_{n+1}(k)}.
\end{aligned}
\end{equation*}
The above is then equal to

\begin{equation*}
\frac{\tau_{m,n}}{m}\cdot\frac{J_1(n)}{J_2(n)}\sum_{k=0}^\infty \frac{I_{m,n}^{(2)}(k)K_{m,n}^{(2)}(k)}{a_{n+1}(k)}\cdot\frac{I_{m,n}^{(1)}(k+1)}{I_{m,n}^{(1)}(k+1)} \le \tau_{m,n}^2\kappa\sum_{k=0}^\infty \frac{\varepsilon(m,n)}{a_{n+1}(k)} = \tau_{m,n}^2\kappa\varepsilon(m,n)s(n+1),
\end{equation*}
again by using Lemmas \ref{special_soln_ineq_2} and \ref{special_sol_bnd} and the fact that $1/c_{1,n}(k)$ is less than $\kappa$.   Thus we have

\begin{equation*}
\|X_{m,n}^{22}\|_{HS}^2 \le \tau_{m,n}^2\kappa\varepsilon(m,n)s(n+1).
\end{equation*}
Very similar arguments, using the same steps as above, show that

\begin{equation*}
\|X_{m,n}^{12}\|_{HS}^2 \le \tau_{m,n}^2\kappa\left(\varepsilon(m,n)+\frac{K_{m,n}^{(1)}(\infty)}{K_{m,n}^{(2)}(\infty)}\right)s(n)\quad\textrm{and}\quad \|X_{m,n}^{21}\|_{HS}^2 \le \tau_{m,n}^2\kappa\varepsilon(m,n)s(n+1).
\end{equation*}
Therefore for $m\neq0$ the Hilbert-Schmidt norms are finite for all of the operators.   

When $m=0$ we have:

\begin{equation*}
\|Z_{0,n}\|_{HS}^2 = \sum_{k=0}^\infty \frac{1}{a_{n+1}(k)}\sum_{i=0}^k \left(\prod_{j=i}^{k-1}c_{2,n}(j)\right)^2\frac{1}{a_{n}(i)} \le s(n)s(n+1),
\end{equation*}
since again $c_{2,n}(k)\le 1$. 

By assumption $s(n)$ goes to zero as $n\to\infty$, Lemma \ref{special_soln_ineq_2} implies that $\varepsilon(m,n)\to0$ as $|m|,n\to\infty$, and
the boundary condition given in equation (\ref{bndy_cond}) requires that

\begin{equation*}
\frac{K_{m,n}^{(1)}(\infty)}{K_{m,n}^{(2)}(\infty)}\to 0\quad\textrm{as}\quad |m|\to\infty .
\end{equation*}
uniformly in $n$. Consequently, the Hilbert-Schmidt norms of $X_{m,n}^{\alpha\beta}$ and $Y_{m,n}^{\alpha\beta}$ go to zero as $|m|,n\to\infty$, and the Hilbert-Schmidt norm of $Z^{(0,n)}$ goes to zero as $n\to\infty$, and the proof is finished.
\end{proof}

The following statement is an immediate consequence of the previous theorem, since the parametrix $Q^{(m,n)}$ is comprised of the integral operators estimated in it.

\begin{cor}\label{quant_parametrix_HS}
The parametrix $Q^{(m,n)}$ for $m\neq 0$ and $n\ge0$ is a Hilbert-Schmidt operator and the Hilbert-Schmidt norm of $Q^{(m,n)}$ goes to zero as $|m|,n\to\infty$.   Moreover for $m=0$ the parametrix $Q^{(0,n)}$ is a Hilbert-Schmidt operator and its Hilbert-Schmidt norm goes to zero as $n\to\infty$.
\end{cor}

We can now close out this section by proving the main theorem of this paper, that the Dirac operator defined by equation (\ref{D_Def}) subject to the boundary condition (\ref{bndy_cond}) has a compact inverse.

\begin{proof}(Proof of Theorem \ref{quant_Q_theorem})
It follows from Proposition \ref{Db_parametrix_rel} that inverse $Q$ of the Dirac operator $D$ is (essentially) a direct sum of $Q^{(m,n)}$ and its analogs for ``negative" terms in the Fourier decomposition.   Corollary \ref{quant_parametrix_HS} shows that $Q^{(m,n)}$ is a Hilbert-Schmidt operator for all $m\in\Z$ and $n\ge0$.   Moreover the same corollary stipulates that the Hilbert-Schmidt norms of $Q^{(m,n)}$ go to zero as $|m|,n\to\infty$.   This means that since $Q$ is a direct sum of compact operators with norms going to zero, which implies that $Q$ is a compact operator.     Thus the proof is complete.
\end{proof}

\end{document}